\documentclass[12pt]{amsart}
\usepackage{amsfonts,amssymb,color}
\usepackage[mathscr]{eucal}
\usepackage{amsmath, amsthm}
\usepackage{mathrsfs}
\usepackage{amsbsy}
\usepackage{dsfont}
\usepackage{bbm}
\usepackage{wasysym}
\usepackage[active]{srcltx} 
\input xypic
\xyoption{all}
\begin{document}
\baselineskip = 5mm
\newcommand \ZZ {{\mathbb Z}} 
\newcommand \NN {{\mathbb N}} 
\newcommand \QQ {{\mathbb Q}} 
\newcommand \RR {{\mathbb R}} 
\newcommand \CC {{\mathbb C}} 
\newcommand \PR {{\mathbb P}} 
\newcommand \AF {{\mathbb A}} 
\newcommand \uno {{\mathbbm 1}}
\newcommand \Le {{\mathbbm L}}
\newcommand \bcA {{\mathscr A}}
\newcommand \bcB {{\mathscr B}}
\newcommand \bcC {{\mathscr C}}
\newcommand \bcD {{\mathscr D}}
\newcommand \bcE {{\mathscr E}}
\newcommand \bcF {{\mathscr F}}
\newcommand \bcM {{\mathscr M}}
\newcommand \bcS {{\mathscr S}}
\newcommand \bcT {{\mathscr T}}
\newcommand \bcU {{\mathscr U}}
\newcommand \bcX {{\mathscr X}}
\newcommand \bcY {{\mathscr Y}}
\newcommand \iHom {{\underline {\rm Hom}}}
\newcommand \Spec {{\rm {Spec}}}
\newcommand \Pic {{\rm {Pic}}}
\newcommand \Alb {{\rm {Alb}}}
\newcommand \Corr {{Corr}}
\newcommand \Sym {{\rm {Sym}}}
\newcommand \LSym {L{\rm {Sym}}}
\newcommand \Alt {{\rm {Alt}}}
\newcommand \cha {{\rm {char}}}
\newcommand \tr {{\rm {tr}}}
\newcommand \ttt {{\rm {t}}}
\newcommand \im {{\rm im}}
\newcommand \Hom {{\rm Hom}}
\newcommand \colim {{{\rm colim}\, }} 
\newcommand \Hocolim {{{\rm Hocolim}\, }} 
\newcommand \hocolim {{{\rm hocolim}\, }} 
\newcommand \coeq {{{\rm coeq}\, }} 
\newcommand \End {{\rm {End}}}
\newcommand \coker {{\rm {coker}}}
\newcommand \id {{\rm {id}}}
\newcommand \Id {{\rm {Id}}}
\newcommand \Ob {{\rm Ob}}
\newcommand \Cyl {{\rm Cyl\; }}
\newcommand \map {{\rm {map}}}
\newcommand \Map {{\rm {Map}}}
\newcommand \op {{^{\rm op}}} 
\newcommand \cor {{\rm {cor}}}
\newcommand \res {{\rm {res}}}
\newcommand \cone {{\rm {cone}}}
\newcommand \psht {{\rm {psht}}}
\newcommand \Spt {{\rm Spt}}
\newcommand \Ev {{\rm Ev}}
\newcommand \Kom {{\rm Kom}}
\newcommand \Tot {{\rm Tot}}
\newcommand \dom {{\rm dom}}
\newcommand \codom {{\rm codom}}
\newcommand \Ord {{\it Ord}}
\newcommand \mg {{\mathfrak m}}
\newcommand \sg {{\Sigma }}
\newcommand \CHM {{\mathscr C\! \mathscr M}}
\newcommand \DM {{\mathscr D\! \mathscr M}}
\newcommand \MM {{\mathscr M\! \mathscr M}}
\newcommand \pt {{*}} 
\newcommand \cf {{\rm cf}} %
\newcommand \SH {{SH}} %
\newcommand \MSH {{\bf SH}} %
\newcommand \MH {{\bf H}} %
\newcommand \de {{\Delta }} 
\newcommand \deop {{\Delta }^{op}} 
\newcommand \SSets {{\deop \mathscr S\! ets}} 
\newcommand \Sets {{\mathscr S\! ets}} 
\newcommand \Sm {{\mathscr S\! m}} 
\newcommand \Sch {{\mathscr S\! ch}} 
\newcommand \Gm {{{\mathbb G}_{\rm m}}}
\newcommand \Top {{\mathscr T\! op}} 
\newcommand \Ke {{\mathscr K}} 
\newcommand \CW {{\mathscr C\! \mathscr W}} 
\newcommand \Shv {{\mathscr S\! hv}} 
\newcommand \cX {{\mathcal X}}
\newcommand \cY {{\mathcal Y}}
\def\r{\color{red}}
\def\b{\color{blue}}
\def\g{\color{green}}
\newtheorem{theorem}{Theorem}
\newtheorem{lemma}[theorem]{Lemma}
\newtheorem{corollary}[theorem]{Corollary}
\newtheorem{proposition}[theorem]{Proposition}
\newtheorem{remark}[theorem]{Remark}
\newtheorem{definition}[theorem]{Definition}
\newtheorem{conjecture}[theorem]{Conjecture}
\newtheorem{example}[theorem]{Example}
\newcommand \lra {\longrightarrow}
\newcommand \hra {\hookrightarrow}
\newenvironment{pf}{\par \noindent{\em Proof}.}{\hfill \framebox(6,6) \par \medskip}
\title[Symmetric powers in abstract homotopy categories]
{\bf Symmetric powers in abstract homotopy categories}
\author{S. Gorchinskiy, V. Guletski\u \i }

\date{24 March 2014}   



\begin{abstract}
\noindent We study symmetric powers in the homotopy categories of abstract closed symmetric monoidal model categories, in both unstable and stable settings. As an outcome, we prove that symmetric powers preserve the Nisnevich and \'etale homotopy type in the unstable and stable motivic homotopy theories of schemes over a base. More precisely, if $f$ is a weak equivalence of motivic spaces, or a weak equivalence between positively cofibrant motivic spectra, with respect to the Nisnevich or \'etale topology, then all symmetric powers $\Sym ^n(f)$ are weak equivalences too. This gives left derived symmetric powers in the corresponding motivic homotopy categories of schemes over a base, which aggregate into a categorical $\lambda $-structures on these categories.
\end{abstract}






\maketitle

\tableofcontents

\section{Introduction}
\label{s-intro}

In topology, symmetric powers preserve homotopy type of $CW$-complexes, which is at the very heart of the fundamental Dold-Thom theorem connecting the homology of a complex to the homotopy groups of its infinite symmetric product. A natural question is to which extent such phenomena could be true in the motivic $\AF ^1$-homotopy theory of schemes over a base? The first steps in this direction were made in the pioneering work \cite{SuslinVoevodsky}. In \cite{VoevMEM} Voevodsky developed a motivic theory of symmetric powers, good enough to construct motivic Eilenberg-MacLane spaces needed for the proof of the Bloch-Kato conjecture. His symmetric powers depend on symmetric powers of schemes presenting motivic spaces. The aim of this paper is to develop a purely homotopical theory of symmetric powers in an abstract symmetric monoidal model category, and to give an affirmative answer to the question when symmetric powers preserve weak equivalences in such a category, working out the unstable and stable settings separately.

More technically, working in a closed symmetric monoidal model category $\bcC $, we address the following two fundamental questions in the paper. Whether left derived symmetric powers exist in the homotopy category $Ho(\bcC )$ and, if they do, whether they aggregate into a (categorical) $\lambda $-structure on the homotopy category of $\bcC $? The latter concept means that, given a morphism in $Ho(\bcC )$, there exists a tower connecting the derived symmetric powers of the domain and codomain, whose cones can be computed by the K\"unneth rule. A categorical $\lambda $-structure is then a system of K\"unneth towers, functorial on morphisms in $Ho(\bcC )$. The decategorification of a categorical $\lambda $-structure via Waldhausen's $K$-theory gives a usual $\lambda $-structure in a commutative ring.

We develop a general machinery to deal with that kind of questions in $\bcC $, and in symmetric spectra over $\bcC $. The methods for the stable and unstable cases are surprisingly different. In the unstable setting, we introduce the notion of symmetrizable cofibrations and study how symmetrizability behaves under cofibrant generation and localization in the sense of \cite{Hirsch}. With this aim, we provide quite a general condition on a left derived functor so that it factors through the localized homotopy category. The main type of localization is the contraction of a diagonalizable interval in $\bcC $. In the stable setting we construct a positive model structure on the category of symmetric spectra, in which weak equivalences are the usual stable weak equivalences and all cofibrations are isomorphisms on level zero. This is a generalization of the positive model structures in topology, constructed by Elmendorf, Kriz, Mandell and May in~\cite{EKMM}. The positive model structure is the main tool in the study of symmetric powers of symmetric spectra over $\bcC $.

Our main destination is, however, the motivic $\AF ^1$-homotopy theory of schemes, and we anticipate numerous applications of our methods and results in arithmetic and geometry through that theory. For example, positive model structures have their account in motivic commutative ring spectra, \cite{CD}, \cite{Hornbostel}, as they do in topology, \cite{Shipley}. For the present, we prove the following two theorems giving positive answers to the above questions in the unstable and stable motivic homotopy theory of schemes over a base:

\medskip

\begin{itemize}

\item[]{}

{\it {\sc Theorem A.} Symmetric powers preserve the Nisnevich and \'etale homotopy type of motivic spaces, left derived symmetric powers exist in the unstable motivic homotopy category of schemes over a base and aggregate into a categorical $\lambda $-structure on it.
}

\end{itemize}

\medskip

\begin{itemize}

\item[]{}

{\it {\sc Theorem B.}
Symmetric powers preserve stable weak equivalences between positively cofibrant motivic symmetric spectra, left derived symmetric powers exist in the motivic stable homotopy category of schemes over a base and aggregate into a categorical $\lambda $-structure on it. The left derived symmetric powers of motivic spectra coincide with the corresponding homotopy symmetric powers.
}

\end{itemize}

\medskip

Now we give a road map of the theorems and propositions appearing in the text. We start the paper by introducing the notion of symmetrizable (trivial) cofibrations in $\bcC $. To study left derived symmetric powers, it would be natural to consider (trivial) cofibrations whose symmetric powers are again (trivial) cofibration. However, we need to introduce a stronger property so that it becomes invariant under compositions and push-outs. Loosely speaking, (trivial) cofibrations are symmetrizable if they are stable under taking colimits of the action of symmetric groups on their push-out products in $\bcC $. If cofibrations are symmetrizable, then one can associate, to a cofibre triangle in $\bcC $, a tower of cofibrations connecting symmetric powers of the vertices of the triangle, and whose cones can be computed by K\"unneth's rule. Such K\"unneth towers can be viewed as a sort of categorification of $\lambda $-structures in commutative rings, and give a powerful tool to work out symmetric powers (Theorem \ref{kuennethkey}). If trivial cofibrations between cofibrant objects are symmetrizable, then symmetric powers preserve weak equivalences between cofibrant objects and so admit their left derived endofunctors on $\bcC $ (Theorem \ref{derived}). When $\bcC $ is cofibrantly generated by the set of generating cofibrations $I$ and the set of trivial generating cofibrations $J$, and if the sets $I$ and $J$ are both symmetrizable, then all cofibrations and trivial cofibrations in $\bcC $ are symmetrizable (Theorem \ref{F-GSW} and Corollary \ref{cofad}). This is useful in applications to concrete cofibrantly generated monoidal model categories, and will be applied to symmetric spectra in Section~\ref{rybakit}. If, in addition, symmetric powers of cofibrant replacements of morphisms in a set of morphisms $S$ are $S$-local equivalences, then trivial cofibrations between cofibrant objects in the left localization $\bcC _S$ are symmetrizable (Theorem \ref{localization}). To show this, we give a condition on a left derived functor (which might not have right adjoint) to factor it through the localized homotopy category (Theorem~\ref{Flocalization}). This result can be applied to a broad range of Bousfield localizations. An important particular case is when $S$-localization is a contraction of a diagonalizable interval into a point (Theorem \ref{locdiagint}).

In topology, i.e. when $\bcC $ is the category of simplicial sets, all cofibrations and trivial cofibrations are symmetrizable (Proposition \ref{symtop1}). If $\bcC $ is the unstable model category of motivic spaces over a base, i.e. the model category for the unstable $\AF ^1$-homotopy category of schemes, cofibrations come up from the simplicial side, so that they are symmetrizable too. The $\AF ^1$-localization is a crux, and Theorem \ref{locdiagint} gives that symmetrizability of trivial cofibrations is stable under $\AF ^1$-localization. In turn, this gives that trivial cofibrations between motivic spaces are symmetrizable, so that the above Theorem \ref{kuennethkey} and Theorem \ref{derived} are applicable in the motivic unstable homotopy theory of schemes over a base. Collecting all these things together we obtain the above Theorem A (Theorem \ref{Teorema A} in the text).

In the stable world, the approach is different. In this paper, a stable homotopy category is the homotopy category of the category $\bcS $ of symmetric spectra over a closed symmetric monoidal model category $\bcC $, stabilizing a smash-with-$T$ functor for a cofibrant object $T$ in $\bcC $, see \cite{Hovey2}. The symmetricity of spectra is very essential due to the fact that the monoidal product in $\bcS $ is coherent with the monoidal product in the corresponding homotopy category $Ho(\bcS )$, \cite{HSS}, \cite{Hovey2}. There are two crucial ingredients in working out symmetric powers of symmetric spectra. The first one is the existence and construction of the positive stable model structure for abstract symmetric spectra (Theorem \ref{positive}). The second ingredient is that $n$-th monoidal powers of positively cofibrant spectra are positively level-wise $\Sigma _n$-equivariantly cofibrant (Proposition \ref{mainresult1}). Using these results we prove (Theorem \ref{mainresult2}) a pretty general version of the theorem due to Elmendorf, Kriz, Mandell and May saying that the $n$-th symmetric power of a positively cofibrant topological spectrum is stably equivalent to the $n$-th homotopy symmetric power of that spectrum, see \cite{EKMM}, Chapter III, Theorem 5.1, and \cite{MMSS}, Lemma 15.5. Our method, however, is different from the one in loc.cit. In constructing positive model structures we systematically use Hirshhorn's localization and in proving Theorem \ref{mainresult2} we use Theorem~\ref{F-GSW} on the stability of symmetrizable (trivial) cofibration under cofibrant generation. Theorem \ref{mainresult2} implies that symmetric powers preserve positive and stable weak equivalences between cofibrant objects in the positive model structure in $\bcS $ (Corollary \ref{hilbre}). In one turn, this gives $\lambda $-structure of left derived symmetric powers in the stable homotopy category $Ho(\bcT )$ (Corollary \ref{main3}). Notice also that the left derived symmetric powers of symmetric spectra are canonically isomorphic to the corresponding homotopy symmetric powers. Now, applying the above general results for symmetric spectra to motivic symmetric spectra of schemes over a base, we obtain Theorem B (Theorem \ref{Teorema B} below).

\medskip

{\it Acknowledgements.} The authors are grateful to Peter May for helpful suggestions and stimulating interest to this work. We also thank Bruno Kahn, Dmitry Kaledin, Roman Mikhailov, Oliver R\"ondigs, Alexander Smirnov, Vladimir Voevodsky and Chuck Weibel for useful discussions along the theme of this paper. The research is done in the framework of the EPSRC grant EP/I034017/1 ``Lambda-structures in stable categories". The first named author was partially supported by the grants RFBR 13-01-12420, NSh-2998.2014.1, AG Laboratory NRU HSE, RF government grant 11.G34.31.0023, and by Dmitry Zimin's Dynasty Foundation. The second named author is grateful to the Institute for Advanced Study in Princeton for the support and hospitality in 2004 - 2006, where and when this work began.

\medskip

\section{Preliminary results}
\label{simplicial}

To get started we recall the notion of a closed symmetric monoidal model category $\bcC $. If $\bcC $ is moreover simplicial then we prove that symmetric powers preserve left homotopies between morphisms, Lemma~\ref{Key Lemma}, motivating all the subsequent considerations in the paper. Being model, $\bcC $ is a category with three distinct pair-wise compatible structures: it is a model category, it is a closed symmetric monoidal category, and it is a simplicial category. Being model, $\bcC $ is equipped with three classes of morphisms, weak equivalences, fibrations and cofibrations, which have the standard lifting properties and meanings, see \cite{Quillen} (or \cite{Hovey1} for modern reference). The monoidality of $\bcC $ means that we have a functor $\wedge :\bcC \times \bcC \to \bcC $ sending any ordered pair of objects $X,Y$ into their monoidal product $X\wedge Y$, and that product is symmetric, i.e. there exists a functorial transposition isomorphism $X\wedge Y\simeq Y\wedge X$. Moreover, the product $\wedge $ is also functorially associative, and there exists a unit object $\uno $, such that $\uno \wedge X\simeq X$ and $X\wedge \uno \simeq X$ for any $X$ in $\bcC $. The monoidal product could be also denoted by $\otimes $ but we prefer to keep to the ``pointed'' notation $\wedge $. Coproducts will be denoted by $\vee $.

A substantial thing here is that monoidality has to be compatible with the model structure. Namely, let $f:X\to Y$ and $f':X'\to Y'$ be two morphisms in $\bcC $ and let
  $$
  (X\wedge Y')\vee _{X\wedge X'}(Y\wedge X')
  $$
be the colimit of the diagram
  $$
  \diagram
  X\wedge X' \ar[rr]^-{f\wedge \id }
  \ar[dd]^-{\id \wedge f'}
  & & Y\wedge X' \\ \\
  X\wedge Y'
  \enddiagram
  $$
A push-out product of $f$ and $f'$ is, by definition, the
unique map
  $$
  f\square f':(X\wedge Y')\vee _{X\wedge X'}(Y\wedge X')\lra
  Y\wedge Y'
  $$
induced by the above colimit. The relation between the model and monoidal structures can be expressed by the following axioms, see Section 4.2 in \cite{Hovey1}:

\medskip

\begin{itemize}

\item[{\rm (A1)}]
If $f$ and $f'$ are cofibrations then $f\square f'$ is also a cofibration. If, in addition, one of the maps $f$ and $f'$ is a weak equivalence, then so is $f\square f'$.

\item[{\rm (A2)}]
If $q:Q\uno \to \uno $ is a cofibrant replacement for the unit object $\uno $, then the maps $q\wedge \id :Q\uno \wedge X\to \uno \wedge X$ and $\id \wedge q:X\wedge Q\uno \to X\wedge \uno $ are weak equivalences for all cofibrant $X$.

\end{itemize}

The functor $X\wedge-$ has right adjoint functor $\iHom (X,-)$. It follows that $X\wedge-$ commutes with colimits.

Now, for any natural number $n$ let $\Sigma _n$ be the symmetric group of permutations of $n$ elements, considered as a category with single object and morphisms being elements of the group. Given an object $X$ in $\bcC $ we have a functor from $\Sigma _n$ to $\bcC $ sending the unique object in $\Sigma_n$ into $X^{\wedge n}$, and permuting factors using the commutativity and associativity constrains in $\bcC $. The $n$-th symmetric power $\Sym ^n(X)$ of $X$ is a colimit of this functor. Clearly, $\Sym ^n$ is an endofunctor on $\bcC $.

\begin{lemma}
\label{Key Lemma}
Suppose that $\bcC$ is, in addition, simplicial and the functor $K\mapsto \uno \wedge K$ from simplicial sets to $\bcC $ is symmetric monoidal. Let $f,g:X\rightrightarrows Y$ be two morphisms in $\bcC $ which are left homotopic, i.e. there exists a morphism $H:X\wedge \Delta [1]\to Y$, such that $H_0=f$ and $H_1=g$, where $\Delta [1]$ is the simplicial interval in $\SSets $. Then, for any natural $n$, the morphism $\Sym ^n(f)$ is left homotopic to the morphism $\Sym ^n(g)$.
\end{lemma}

\begin{pf}
Let $\delta_n :\Delta [1]\to \Delta [1]^{\wedge n}$ be the diagonal morphism and $\alpha _n$ be the composition of $\id _{X^{\wedge n}}\wedge \delta _n$ with the isomorphism between $X^{\wedge n}\wedge \Delta[1]^{\wedge n}$ and $(X\wedge \Delta[1])^{\wedge n}$. Then $H^{\wedge n}\circ \alpha _n:X^{\wedge n}\wedge \Delta [1]\to Y^{\wedge n}$ is a left homotopy between $f^{\wedge n}$ and $g^{\wedge n}$. The cylinder functor $-\wedge \Delta [1]$ has right adjoint, so commutes with colimits. As permuting factors does not effect the diagonal, the permutation of factors in $(X\wedge \Delta [1])^{\wedge n}$ is coherent with the permutation of factors in $X^{\wedge n}$ in the product $X^{\wedge n}\wedge \Delta [1]$. Taking colimits over $\Sigma _n$ we obtain a left homotopy between $\Sym ^n(f)$ and $\Sym ^n(g)$.
\end{pf}

\begin{example}
\label{examp-compl}
{\rm The existence of a simplicial structure, and its compatibility with the symmetric monoidal structure on $\bcC$ in Lemma \ref{Key Lemma} are essential. Indeed, let $\Kom(\ZZ)$ be the category of unbounded complexes of abelian groups. The category $\Kom(\ZZ)$ inherits the symmetric monoidal structure via total complexes $\Tot (-\otimes -)$ and has a natural structure of a model category whose weak equivalences are quasi-isomorphisms and fibrations are term-wise epimorphisms. The category is closed symmetric monoidal model but not simplicial. Let $X$ be the complex $\ldots \to 0\to \ZZ \stackrel{\id }{\lra }\ZZ \to 0\to \ldots $, where $\ZZ $ is concentrated in degrees $-1$ and $0$ respectively. This complex is homotopically trivial. On the other hand, a calculation shows that $\Sym ^2(X)$ is the complex
  $$
  \ldots \to 0\to \ZZ /2\stackrel{0}{\lra }\ZZ \stackrel{\id }{\lra }\ZZ \to 0\to \ldots \; ,
  $$
where $\ZZ /2$ stands in degree $-2$. Clearly, this $\Sym ^2(X)$ has non-trivial cohomology group in degree $-2$.
}
\end{example}

Let now $\bcC $ be as in Lemma \ref{Key Lemma}, and let $Ho(\bcC )$ be the homotopy category of $\bcC $. A naive way to define symmetric powers in $Ho(\bcC )$ would be through Lemma \ref{Key Lemma} and the standard treatment of homotopy categories as subcategories of fibrant-cofibrant objects factorized by left homotopies on Hom-sets, see \cite[1.2]{Hovey1} or \cite{Quillen}. Indeed, let $\bcC _{\cf }$ be the full subcategory of objects which are fibrant and cofibrant simultaneously. Let, furthermore, $ho(\bcC )$ be the quotient category of $\bcC _{\cf }$ by left homotopic morphisms between fibrant-cofibrant objects in $\bcC $. As symmetric powers respect left homotopies by Lemma \ref{Key Lemma}, we have now a functor $\Sym ^n:ho(\bcC )\to Ho(\bcC )$. The category $\bcC $, being a model category, is endowed with a fibrant replacement functor $R:\bcC \to \bcC _{\rm f}$ and a cofibrant replacement functor $Q:\bcC \to \bcC _{\rm c}$. Combining both we obtain mixed replacement functors $RQ$ and $QR$ from $\bcC $ to the full subcategory $\bcC _{\cf }$ of fibrant-cofibrant objects in $\bcC $, any of which induces a quasi-inverse to the obvious functor from $ho(\bcC )$ to $Ho(\bcC )$. Then one might wish to construct an endofunctor $\Sym ^n$ on $Ho(\bcC )$ as a composition of this quasi-inverse and the above $\Sym ^n$. However, this way is not too much natural as we still have no derived symmetric powers in $Ho(\bcC )$.

\section{Symmetrizable cofibrations}
\label{symcofib}

In this section, we introduce the notion of symmetrizable (trivial) cofibrations. The main result, Theorem~\ref{F-GSW}, asserts that this property is stable under the operations participating in cofibrant generation of model categories. This gives that to check symmetrizability of (trivial) cofibrations it is enough to examine it on generating (trivial) cofibrations, see Corollaries~\ref{cofad},~\ref{symcof} and~\ref{Fsymcof}.

Let $\bcC $ be a closed symmetric monoidal model category with the monoidal product $\wedge :\bcC \times \bcC \to \bcC $. For any two morphisms $f:X\to Y$ and $f':X'\to Y'$ in $\bcC $, let
  $$
  \Box (f,f')=(X\wedge Y')\vee _{X\wedge X'}(Y\wedge X')
  $$
be the coproduct over $X\wedge X'$, and let
  $$
  f\square f':\Box (f,f')\lra Y\wedge Y'
  $$
be a universal morphism from the colimit into $Y\wedge Y'$.
Such defined $\Box $-operations are commutative and associative in the obvious sense. For example, for any three morphisms $f:X\to Y$, $f':X'\to Y'$ and $f'':X''\to Y''$ in $\bcC $ the morphism $(f\Box f')\Box f''$ is the same as the morphism $f\Box (f'\Box f'')$ up to the canonical isomorphism between $\Box (f\Box f',f'')$ and $\Box (f,f'\Box f'')$. Since $\Box $ is an associative operation, for any finite collection $f_i:X_i\to Y_i$, $i=1,\dots ,l$, of morphisms in $\bcC $ we have a well defined morphism
  $$
  f_1\Box \dots \Box f_l:\Box (f_1,\dots ,f_l)\lra Y_1\wedge \dots \wedge Y_l\; .
  $$

For simplicity, let $X'=X$, $Y'=Y$ and $f'=f$. Then we have the $\Box $-squares $\Box ^2_1(f)=\Box (f,f)$ and $f^{\Box 2}=f\Box f$, which can be generalized for higher degrees as follows. Let $\Gamma $ be the category with two objects $0$ and $1$ and one morphism $0\to 1$, and let $\Gamma ^n$ be the $n$-fold Cartesian product of $\Gamma $ with itself. Objects in $\Gamma ^n$ are ordered $n$-tuples of zeros and units. A functor $K:\Gamma \to \bcC $ is just a morphism $f:X\to Y$ in $\bcC $. It is also natural to write $K(f)$ rather than $K$, since $K$ is fully determined by the morphism $f$. For any natural $n$ let $K^n$ be the composition of the $n$-fold Cartesian product $\Gamma ^n\to \bcC ^n$ and the functor $\bcC ^n\stackrel{\wedge }{\to }\bcC $. For any $0\leq i\leq n$ one has a full subcategory $\Gamma ^n_i$ in $\Gamma ^n$ generated by $n$-tuples having not more than $i$ units in them. The restriction of $K^n$ on $\Gamma ^n_i$ will be denoted by $K^n_i(f)$, or simply by $K^n_i$ when $f$ is clear. In other words, $K^n_i$ is a subdiagram in $K^n$ having not more than $i$ factors $Y$ in each vertex. Let then
  $$
  \Box ^n_i(f)=\colim K^n_i(f)
  $$
or simply
  $$
  \Box ^n_i=\colim K^n_i\; .
  $$
Since $K^n_0=X^n$ and $K^n_n=K^n$, respectively, $\Box ^n_0=X^n$ and $\Box^n_n=Y^n$. As $K^n_{i-1}$ is a subdiagram in $K^n_i$ one has a morphism on colimits
  $$
  \Box ^n_{i-1}\lra \Box ^n_i
  $$
for any $1\leq i\leq n$.

Suppose $\bcC$ is cofibrantly generated. Let $G$ be a finite group considered as a one-object category, and let $\bcC ^G$ be the category of functors from $G$ to $\bcC $. We shall be using the standard model structure on $\bcC ^G$ provided by Theorem 11.6.1 in \cite{Hirsch}. In particular, given a morphism $f$ in $\bcC ^G$, it is a weak equivalence (fibration) in $\bcC ^G$ if and only if the same $f$, as a morphism in $\bcC $, is a weak equivalence (fibration) in $\bcC $. For any object $X$ in $\bcC ^G$, let $X/G$ be the colimit of the action of the group $G$ on $X$. \label{See C1} This is a functor from $\bcC ^G$ to $\bcC $ preserving cofibrations.

The group $\sg _n$ acts on $\Gamma ^n$ and so on $K^n$. Each subcategory $\Gamma ^n_i$ is invariant under the action of $\sg _n$. Then $\sg _n$ acts on $K^n_i$ and so on $\Box ^n_i$. Let
  $$
  \tilde \Box ^n_i(f)=\colim _{\sg _n}\Box ^n_i(f)
  $$
for each index $i$. Obviously, $\tilde \Box ^n_0(f)=\Sym ^n(X)$
and $\tilde \Box ^n_n(f)=\Sym ^n(Y)$, and for each index $i$ we have a universal morphism between colimits
  $$
  \tilde \Box ^n_{i-1}(f)\lra \tilde \Box ^n_i(f)\; .
  $$
Sometimes we will drop the morphism $f$ from the notation writing
  $$
  \tilde \Box ^n_i=\colim _{\sg _n}\Box ^n_i\; ,
  $$
  $$
  \tilde \Box ^n_{i-1}\lra \tilde \Box ^n_{i-1}\; ,
  $$
etc.

In new notation, the axiom (A1) of a monoidal model category says, in particular,  that for any cofibration $f:X\to Y$ in $\bcC $ the push-out product
  $$
  f^{\Box 2}:\Box ^2_1(f)\lra Y\wedge Y
  $$
is also a cofibration in $\bcC $. By associativity, it implies that the morphism
  $$
  f^{\Box n}:\Box ^n_{n-1}(f)\lra Y^{\wedge n}
  $$
is a cofibration in $\bcC $ for any natural $n$, not only for $n=2$. It doesn't mean, of course, that the $\Sigma _n$-equivariant morphism $f^{\Box n}$ is a cofibration in $\bcC ^{\Sigma _n}$. A morphism $f:X\to Y$ in $\bcC $ is said to be a {\it symmetrizable} (trivial) cofibration if the corresponding morphism
  $$
  f^{\tilde \Box n}:\tilde \Box ^n_{n-1}(f)\lra \Sym ^n(Y)
  $$
is a (trivial) cofibration for any integer $n\ge 1$. A symmetrizable (trivial) cofibration $f$ is itself a (trivial) cofibration because $\tilde \Box ^1_0(f)\to \Sym ^1(Y)$ is nothing but the original morphism $f$.

Notice that symmetrizability of (trivial) cofibrations is not always the case. Example~\ref{examp-compl} shows that trivial cofibrations are not symmetrizable in the category $\Kom(\ZZ )$. As we will show later, see Remark~\ref{remark-nosymm}, cofibrations are not symmetrizable for spectra over simplicial sets.

\begin{remark}
\label{nostrong}
{\rm If $f:X\to Y$ is a symmetrizable (trivial) cofibration in $\bcC $ it is not necessarily true that the $\Sigma _n$-equivariant morphism $f^{\Box n}$ is a cofibration in $\bcC ^{\Sigma _n}$. Formally, it would also make sense to say that $f$ is a {\it strongly symmetrizable} (trivial) cofibration if $f^{\Box n}$ is a (trivial) cofibration in $\bcC ^{\Sigma _n}$. However, such defined strongly symmetrizable cofibrations are not of much use to us because, as the following example shows, they do not occur even in topology.}
\end{remark}

\begin{example}
\label{nostrongex}
{\rm Let $\bcC $ be the model category of simplicial sets $\SSets $. According to our notation, $\wedge $ in this $\bcC $ stands for the usual Cartesian product of simplicial sets. Let $E\Sigma _n$ be the contractible simplicial set with $(E\Sigma _n)_i=\Sigma _n^{\times i}$, and the diagonal action of $\Sigma _n$. The morphism $f:\emptyset \to X$ is a cofibration for any simplicial set $X$. Then the morphism $f^{\tilde \Box n}$ from $\emptyset $ to $\Sym ^n(X)$ is also a cofibration, for any $n\geq 1$. However, the morphism $f^{\Box n}$ from $\emptyset $ to $X^{\wedge n}$ is not a cofibration in $\bcC ^{\Sigma _n}$. The reason is in fixed points of $X^{\wedge n}$ given by the diagonal map from $X$ to $X^{\wedge n}$. Fixed points have the effect that there are no $\Sigma _n$-morphisms from $X^{\wedge n}$ to $E\Sigma_n\wedge X^{\wedge n}$, as $\Sigma _n$ acts term-wise freely on the simplicial set $E\Sigma _n\wedge X^{\wedge n}$. It follows that the morphism $f^{\Box n}$ does not have a $\Sigma _n$-left lifting property with respect to the trivial fibration $E\Sigma _n\wedge X^{\wedge n}\to X^{\wedge n}$ in $\bcC ^{\Sigma _n}$ and the identity map from $X^{\wedge n}$ to itself. A similar argument applies in many other cases, for example, for pointed simplicial sets and motivic spaces.}
\end{example}

Symmetrizable cofibrations have their own disadvantages. Namely, they do serve unstable homotopy categories, but do not work in stable homotopy theories, since cofibrations between spectra are typically not symmetrizable, see Remark~\ref{remark-nosymm} below. This is why we shall give a better definition, which will serve all the needs relevant to symmetric powers in stable categories.

Let $\bcC $ be a closed symmetric monoidal model category, let $\bcD $ be a cofibrantly generated model category, and let $F:\bcC \to \bcD $ be a functor from $\bcC $ to $\bcD $. Then $F$ induces a functor from $\bcC ^G$ to $\bcD ^G$, which will be denoted by the same symbol $F$. A finite collection $\{ n_1,\dots ,n_l\} $ of non-negative integers will be called a multidegree. A class of morphisms $M$ in $\bcC $ will be called a {\it symmetrizable class} of (trivial) $F$-cofibrations in $\bcC $ if for any finite collection $\{ f_1,\dots ,f_l\} $ of morphisms in the class $M$ and any multidegree $\{ n_1,\dots ,n_l\} $ the morphism
  $$
  F(f_1^{\tilde \Box n_1}\Box \dots \Box f_l^{\tilde \Box n_l})
  $$
is a (trivial) cofibration in the model category $\bcD $. Respectively, the class $M$ will be called a {\it strongly symmetrizable class} of (trivial) $F$-cofibrations in $\bcC $ if for any finite collection $\{ f_1,\dots ,f_l\} $ of morphisms in $M$ and any multidegree $\{ n_1,\dots ,n_l\} $ the morphism
  $$
  F(f_1^{\Box n_1}\Box \dots \Box f_l^{\Box n_l})
  $$
is a (trivial) cofibration in the model category $\bcD ^{\Sigma _{n_1}\times \dots \times \Sigma _{n_l}}$.

Notice that if $\bcD =\bcC $ and $F$ is the identity functor, then $M$ is a (strongly) symmetrizable class of (trivial) $\Id $-cofibrations if and only if $M$ consists of (strongly) symmetrizable (trivial) cofibrations in $\bcC $. The case $l>1$ is essential when $F$ is not monoidal. This will hold in the applications to symmetric spectra in Section \ref{rybakit}.

Let now $\lambda $ be an ordinal and let $X$
be a functor from $\lambda $ to a model category $\bcC $ preserving colimits (although $\lambda $ is not necessarily cocomplete). To shorten notation, for any ordinal $\alpha <\lambda $ let $X_{\alpha }$ be the object $X(\alpha )$, and for any two ordinals $\alpha $ and $\beta $, such that $\alpha \leq \beta <\lambda $, let $f_{\beta ,\alpha }=X(\alpha \leq \beta )$. Let also $X_{\infty }=\colim (X)$ and, for any ordinal $\alpha <\lambda $, let $f_{\infty ,\alpha }:X_{\alpha }\to X_{\infty }$ be the canonical morphism into colimit. Since the set of objects in $\lambda $ has the minimal object $0$, we have the canonical morphism $f_{\infty }:X_0\to X_{\infty }$, which is called a transfinite composition induced by the functor $X$.

If $M$ is a class of morphisms in the category $\bcC $ and $f_{\alpha +1,\alpha }$ is from $M$ for any ordinal $\alpha $, such that $\alpha +1<\lambda $, then we say that $f_{\infty }:X_0\to X_{\infty }$ is a transfinite composition of morphisms from $M$.

\begin{theorem}
\label{F-GSW}
Let $\bcC $ be a cofibrantly generated closed symmetric monoidal model category, $F:\bcC \to \bcD $ be a functor from $\bcC $ to a cofibrantly generated model category $\bcD $ commuting with colimits, and let $M $ be a class of morphisms in $\bcC $. Then the property of $M$ to be a (strongly) symmetrizable class of (trivial) $F$-cofibrations is stable under the addition to $M$ of

\begin{itemize}

\item[{\rm (A)}]
a push-out of a morphism from $M $;

\item[{\rm (B)}]
a retract of a morphism from $M $;

\item[{\rm (C)}]
a composition $g\circ f$, where $f$ and $g$ are two composable morphisms from $M $;

\item[{\rm (D)}]
a transfinite composition $f_{\infty }:X_0\to X_{\infty }$ induced by a functor $X:\lambda \to \bcC $, where $\lambda $ is an ordinal, $X$ commutes with colimits, and for any $\alpha <\lambda $, such that $\alpha +1<\lambda $, the morphism $f_{\alpha +1,\alpha }:X_{\alpha }\to X_{\alpha +1}$ belongs to the class $M $.

\end{itemize}

\end{theorem}

\begin{remark}
{\rm Item (C) can be considered as a particular case of item (D). The category $\bcD$ is required to be cofibrantly generated merely to have a model structure on the category $\bcD^{\Sigma_n}$.}
\end{remark}

The proof of Theorem \ref{F-GSW} occupies the next section of the paper. Now we discuss its consequences. Suppose $\bcC$ is cofibrantly generated by a set of generating cofibrations $I$ and a set of generating trivial cofibrations $J$.

\begin{corollary}
\label{cofad}
If $I$ is a (strongly) symmetrizable set of $F$-cofibrations, then the class of all cofibrations in $\bcC $ is a (strongly) symmetrizable class of $F$-cofibrations. Similarly, if $J$ is a (strongly) symmetrizable set of trivial $F$-cofibrations, then the class of all trivial cofibrations in $\bcC $ is a (strongly) symmetrizable class of trivial $F$-cofibrations.
\end{corollary}

\begin{proof}
By Theorem~\ref{F-GSW}, the class of retracts of relative $I$-cell complexes is a (strongly) symmetrizable class of $F$-cofibrations. On the other hand, this class coincides with all cofibrations in $\bcC$. Similar argument applies to trivial cofibrations.
\end{proof}

Applying Corollary \ref{cofad} to a cofibration $\emptyset \to X$ we obtain two more corollaries.

\begin{corollary}
\label{symcof}
Suppose all morphisms in $I$ are symmetrizable. Then any symmetric power $\Sym ^n(X)$ of a cofibrant object $X$ in $\bcC $ is cofibrant.
\end{corollary}

\begin{corollary}
\label{Fsymcof}
If $I$ is a strongly symmetrizable set of $F$-cofibrations, then for any cofibrant object $X$ in $\bcC $ we have that $F(X^{\wedge n})$ is a cofibrant object in $\bcD^{\Sigma_n}$.
\end{corollary}

For short, by abuse of notation, throughout the text we will say that $I$ is symmetrizable if it consists of symmetrizable cofibrations, and that $J$ is symmetrizable if it consists of symmetrizable trivial cofibrations.

Finally, we compare the pointed v.s. unpointed cases of our setup. Assuming the terminal object $*$ is the monoidal unit and cofibrant in $\bcC $, the pointed category $\bcC _*= * \downarrow \bcC $ inherits the monoidal model structure by Proposition 4.2.9 in \cite{Hovey1}.

\begin{lemma}
\label{pointed}
Let $f$ be a morphism in $\bcC _*$, which is a symmetrizable (trivial)  cofibration as a morphism in $\bcC $. Then $f$ is a symmetrizable (trivial) cofibration as a morphism in $\bcC _*$.
\end{lemma}

\begin{pf}
This follows from the fact that $f^{\Box n}$ in $\bcC _*$ is a push-out of $f^{\Box n}$ in $\bcC $.
\end{pf}

\section{The proof of Theorem \ref{F-GSW}}
\label{boksikcalculus}

First we collect some technical lemmas needed in proving the theorem. If $f:X\to Y$ is a morphism in $\bcC $ and there exists a push-out square
   $$
  \diagram
  X' \ar[dd]_-{f'} \ar[rr]^-{} & & X \ar[dd]^-{f} \\ \\
  Y' \ar[rr]^-{} & & Y
  \enddiagram
  $$
then sometimes we will write
  $$
  f=\psht (f')\; ,
  $$
not specifying the horizontal morphisms of the square.

\begin{lemma}
\label{ulitka}
Let $f=\psht (f')$, $e:A\to B$ a morphism in $\bcC $ and let
  $$
  d:\Box (f',e)\lra \Box (f,e)
  $$
be the universal morphism between two colimits induced by the push-out square above. Then the commutative square
  $$
  \diagram
  \Box (f',e) \ar[dd]_-{f'\Box e} \ar[rr]^-{d} & & \Box (f,e) \ar[dd]^-{f\Box e} \\ \\
  Y'\wedge B \ar[rr]^-{} & & Y\wedge B
  \enddiagram
  $$
is push-out, i.e. $\psht(f')\Box e=\psht(f'\Box e)$.
\end{lemma}

\begin{pf}
As $\wedge $-multiplication is adjoint from the left and so it commutes with colimits, the commutative squares
  $$
  \diagram
  X'\wedge A \ar[dd]_-{f'\wedge \id } \ar[rr]^-{} & & X\wedge A \ar[dd]^-{f\wedge \id } \\ \\
  Y'\wedge A \ar[rr]^-{} & & Y\wedge A
  \enddiagram\quad\quad
  \diagram
  X'\wedge B \ar[dd]_-{f'\wedge \id } \ar[rr]^-{} & & X\wedge B \ar[dd]^-{f\wedge \id } \\ \\
  Y'\wedge B \ar[rr]^-{} & & Y\wedge B
  \enddiagram
  $$
are push-out. The morphism $e$ induces a morphism from the left push-out square to the right one. This and the universal property of the colimits $\Box(f',e)$ and $\Box(f,e)$ allow to show that the commutative square in question is push-out.
\end{pf}

\begin{lemma}
\label{vasnecov}
Let $f_1,\ldots,f_n$ be a collection of morphisms in $\bcC$. Then we have
$$
\psht(f_1)\Box\ldots\Box\psht(f_n)=\psht(f_1\Box\ldots \Box f_n)\;.
$$
\end{lemma}
\begin{proof}
Use Lemma~\ref{ulitka} and associativity of the $\Box$-product.
\end{proof}

Let $G$ be a finite group and let $H$ be a subgroup in it. The natural restriction $\res ^G_H:\bcC ^G\to \bcC ^H$ has left adjoint functor $\cor ^G_H:\bcC ^H\to \bcC ^G$, such that $(\cor ^G_H,\res ^G_H)$ is a Quillen adjunction, see Theorem 11.9.4 in \cite{Hirsch}. Recall that $\cor ^G_H(X)\simeq (G\times X)/H$ and $\cor^G_H(X)/G\simeq X/H$.

\begin{lemma}
\label{at the core}
Let $X\stackrel{f}{\to }Y\stackrel{g}{\to }Z$ be two composable morphisms in $\bcC $, and let $n$ be a positive integer. Then  the morphism $(gf)^{\Box n}:\Box ^n_{n-1}(gf)\to Z^{\wedge n}$ is a composition
  $$
  g^{\Box n}\circ \psht (\cor ^{\Sigma _n}_{\Sigma _{n-1}\times \Sigma _{1}}(g^{\Box (n-1)}\Box f))\circ
  \dots \circ \psht (\cor ^{\Sigma _n}_{\Sigma _{1}\times \Sigma _{n-1}}(g\Box f^{\Box (n-1)}))\circ \psht (f^{\Box n})\; ,
  $$

\noindent where $\Sigma _{i}\times \Sigma _{n-i}$ is canonically embedded into $\Sigma _n$ for each $i$, and $\psht (f^{\Box n})$ is a push-out of $f^{\Box n}$ with respect to the universal morphism $\Box ^n_{n-1}(f)\to \Box ^n_{n-1}(gf)$. \label{See D2 & D3}
\end{lemma}

\begin{pf}
Let $J$ be the category $0\to 1\to 2$. A pair of subsequent morphisms $f:X\to Y$ and $g:Y\to Z$ in $\bcC $ can be considered as a functor $K(f,g):J\to \bcC $ from $J$ to $\bcC $. Let $J^n$ and $\bcC ^n$ be the Cartesian $n$-th powers of the categories $J$ and $\bcC $ respectively, and let $K^n(f,g):J^n\to \bcC$ be the composition of the $n$-th Cartesian power of the functor $K(f,g)$ and the $n$-th monoidal product $\wedge :\bcC ^n\to \bcC $. In particular, $K^n(f,g)$ is a commutative diagram in $\bcC $, whose vertices are monoidal products of the three objects $X$, $Y$ and $Z$. Notice that the order of the factors is important here.

For short, let $K^n=K^n(f,g)$, and consider a subdiagram $L$ in $K^n$ generated by the vertices containing at least one factor $X$, and for any index $i\in \{ 0,1,\dots ,n\} $ let $K^n_i$ be a subdiagram in $K^n$ generated by vertices containing $\leq i$ factors $Z$. Let also $L_i=L\cup K^n_i$ and put $L_{-1}=L$. Then we have a filtration
  $$
  L_{-1}\subset L_0\subset L_1\subset L_2\subset \dots \subset L_n=K^n
  $$
and, respectively, a chain of morphisms between colimits
  $$
  \colim (L_{-1})\to \colim (L_0)\to \colim (L_1)\to \dots \to
  \colim (K^n)\; ,
  $$
whose composition is nothing but $(gf)^{\Box n}$.

For any $0\leq i\leq n$ the object $\Box (g^{\Box i},f^{\Box (n-i)})$ is a colimit of a subdiagram in $L_{i-1}$, so that one has a universal morphism from $\Box (g^{\Box i},f^{\Box (n-i)})$ to $\colim (L_{i-1})$. Since $Z^{\wedge i}\wedge Y^{\wedge (n-i)}$ is a vertex in the diagram $L_i$, we have a morphism from $Z^{\wedge i}\wedge Y^{\wedge (n-i)}$ to $\colim (L_i)$. Finally, we have a standard morphism $g^{\Box i}\Box f^{\Box (n-i)}:\Box (g^{\Box i},f^{\Box (n-i)})\to Z^{\wedge i}\wedge Y^{\wedge (n-i)}$. Collecting these morphisms together we get a commutative diagram $$
  \diagram
  \Box (g^{\Box i},f^{\Box (n-i)}) \ar[dd]_-{}\ar[rr]^-{}& & Z^{\wedge i}\wedge Y^{\wedge (n-i)} \ar[dd]_-{}\\ \\
  \colim (L_{i-1})   \ar[rr]^-{} & & \colim (L_i)
  \enddiagram
  $$
This is a $\Sigma_{i}\times \Sigma_{n-i}$-equivariant commutative diagram, which yields a $\Sigma_n$-equivariant commutative diagram
  $$
  \diagram
  \cor ^{\Sigma _n}_{\Sigma _{i}\times \Sigma _{n-i}}(\Box (g^{\Box i},f^{\Box (n-i)})) \ar[dd]_-{} \ar[rr]^-{} & & \cor ^{\Sigma _n}_{\Sigma _{i}\times \Sigma _{n-i}}(Z^{\wedge i}\wedge Y^{\wedge (n-i)}) \ar[dd]_-{}\\ \\
  \colim (L_{i-1}) \ar[rr]^-{} & & \colim (L_i)
  \enddiagram
  $$
Straightforward verification shows that this is a push-out square.
\end{pf}

\begin{lemma}
\label{potap}
For any three morphisms $X\stackrel{f}{\to }Y\stackrel{g}{\to }Z$ and $A\stackrel{e}{\to }B$ in $\bcC$ we have that
  $$
  (gf)\Box e=(g\Box e)\circ \varkappa \; ,
  $$
where $\varkappa :\Box (gf,e)\to \Box (g,e)$ is a universal morphism between colimits and the square
  $$
  \diagram
  \Box (f,e) \ar[dd]_-{f\Box e} \ar[rr]^-{} & & \Box (gf,e) \ar[dd]^-{\varkappa } \\ \\
  Y\wedge V \ar[rr]^-{} & & \Box (g,e)
  \enddiagram
  $$
is push-out, i.e. we have $(gf)\Box e=(g\Box e)\circ \psht(f\Box e)$.
\end{lemma}

\begin{pf}
The top horizontal morphism $\Box (f,e)\to \Box (gf,e)$ in the above diagram is also a universal morphism between colimits. The proof of the lemma then follows from the appropriate commutative diagrams for the products $f\Box e$, $gf\Box e$ and $g\Box e$ involved into the lemma. \label{See D1}
\end{pf}

\begin{lemma}
\label{kotofei}
Let $X_1\stackrel{f_1}{\to }X_2\stackrel{f_2}{\to }\dots \stackrel{f_n}{\to }X_{n+1}$ and $e:A\to B$ be morphisms in $\bcC $. Then one has
$$
(f_n\circ \dots \circ f_1)\Box e=(f_n\Box e)\circ\psht(f_{n-1}\Box e)\circ\dots\circ\psht(f_1\Box e)\;.
$$
\end{lemma}

\begin{pf}
Use induction by $n$. If $n=2$ then the lemma is just Lemma \ref{potap}. For the inductive step,
  $$
  (f_n\circ \dots \circ f_1)\Box h=
  (f_n\circ (f_{n-1}\circ \dots \circ f_1))\Box h=
  $$
  $$
  (f_n\Box h)\circ
  \psht ((f_{n-1}\circ \dots \circ f_1)\Box h)\; ,
  $$
where the last equality is provided by Lemma \ref{potap} too.
\end{pf}

\begin{lemma}
\label{doublecor}
Let $G$ and $G'$ be two finite groups, let $H$ be a subgroup in $G$ and $H'$ be a subgroup in $G'$. Let also $f:X\to Y$ and $f':X'\to Y'$ be two morphisms in $\bcC $. Then
  $$
  \cor ^G_H(f)\Box \cor ^{G'}_{H'}(f')=
  \cor ^{G\times G'}_{H\times H'}(f\Box f')\; .
  $$
\end{lemma}

\begin{pf}
The lemma holds true because $\wedge $-multiplication commutes with colimits and the order of counting colimits is not important.
\end{pf}

\begin{lemma}
\label{aux}
Let $\lambda $ be an ordinal. For any two functors $X$ and $Y$ from $\lambda $ to $\bcC $, not necessarily preserving colimits, one has
  $$
  \colim _{\alpha <\lambda }(X_{\alpha }\wedge Y_{\alpha })=
  (\colim _{\alpha <\lambda }X_{\alpha })\wedge (\colim _{\beta <\lambda }Y_{\beta })
  $$
\end{lemma}

\begin{pf}
Indeed, as the monoidal product $\wedge $ in $\bcC $ is closed, smashing with an object commutes with colimits. This is why
  $$
  (\colim _{\alpha <\lambda }X_{\alpha })\wedge (\colim _{\beta <\lambda }Y_{\beta })=\colim _{\alpha <\lambda }(X_{\alpha }\wedge \colim _{\beta <\lambda }Y_{\beta })=
  $$
  $$
  =\colim _{\alpha <\lambda }\colim _{\beta <\lambda }(X_{\alpha }\wedge Y_{\beta })\; .
  $$
Since all arrows in the diagram $X\wedge Y$ are targeted towards the diagonal objects $X_{\alpha }\wedge Y_{\alpha }$, the last colimit is the colimit of these diagonal objects $X_{\alpha }\wedge Y_{\alpha }$.
\end{pf}

Let now $\lambda $ be an ordinal and let $X$ be a colimit-preserving functor from the ordinal $\lambda $ to the category $\bcC $. For any two ordinals $\alpha $ and $\beta $, such that $\alpha \leq \beta <\lambda $, let $\Box ^n_{n-1}(f_{\alpha ,0})\to \Box ^n_{n-1}(f_{\beta ,0})$ be the universal morphism from Lemma \ref{at the core} being applied to the composition $f_{\beta ,0}=f_{\beta ,\alpha }\circ f_{\alpha ,0}$. Similarly, let $\Box ^n_{n-1}(f_{\alpha ,0})\to \Box ^n_{n-1}(f_{\infty })$ be the universal morphism from Lemma \ref{at the core} applied to the composition $f_{\infty }=f_{\infty ,\alpha }\circ f_{\alpha ,0}$. It is not hard to verify that the collection of objects $\Box ^n_{n-1}(f_{\alpha ,0})$ and morphisms $\Box ^n_{n-1}(f_{\alpha ,0})\to \Box ^n_{n-1}(f_{\beta ,0})$ gives a functor from $\lambda $ to $\bcC $.

\begin{lemma}
\label{3na}
In the above terms,
  $$
  \Box ^n_{n-1}(f_{\infty })=\colim _{\alpha <\lambda }\Box ^n_{n-1}(f_{\alpha ,0})\; ,
  $$
  $$
  X_{\infty }^{\wedge n}=\colim _{\alpha <\lambda }X^{\wedge n}_{\alpha }\; ,
  $$
and the square
  $$
  \diagram
  \Box ^n_{n-1}(f_{\alpha ,0})\ar[dd]_-{f_{\alpha ,0}^{\Box n}} \ar[rr]^-{} & & \Box ^n_{n-1}(f_{\infty }) \ar[dd]^-{f_{\infty }^{\Box n}} \\ \\
  X^{\wedge n}_{\alpha } \ar[rr]^-{} & & X_{\infty }^{\wedge n}
  \enddiagram
  $$
is commutative for any $\alpha $, i.e.
  $$
  f_{\infty }^{\Box n}=\colim _{\alpha <\lambda }(f_{\alpha ,0}^{\Box n})\; .
  $$
\end{lemma}

\begin{pf}
By Lemma \ref{aux},
  $$
  K^n_i(f_{\infty})=\colim _{\alpha <\lambda }K^n_i(f_{\alpha,0})
  $$
for any index $i$, where the colimit is taken in the category of functors from subcategories in $I^n$ to $\bcC $. It implies the following computation:
  $$
  \Box ^n_i(f_{\infty})=\colim K^n_i(f_{\infty})=\colim (\colim _{\alpha <\lambda }K^n_i(f_{\alpha,0}))=
  $$
  $$
  =\colim _{\alpha <\lambda }(\colim K^n_i(f_{\alpha,0}))=
  \colim _{\alpha <\lambda }\Box ^n_i(f_{\alpha,0})\; .
  $$
In particular,
  $$
  \Box ^n_{n-1}(f_{\infty})=\colim _{\alpha <\lambda }\Box ^n_{n-1}(f_{\alpha,0})\; ,
  $$
  $$
  X_{\infty }^n=\Box ^n_n(f_{\infty})=\colim _{\alpha <\lambda }\Box ^n_n(f_{\alpha,0})=
  \colim _{\alpha <\lambda }X^n_{\alpha }\; ,
  $$
and both isomorphisms are connected by the corresponding commutative square.
\end{pf}

\begin{lemma}
\label{ladder}
Let $\bcE $ be a model category and let $\lambda $ be an ordinal. Let
  $$
  \xymatrix{
  U,V:\lambda \ar@<+0.5ex>[r]^-{} \ar@<-0.5ex>[r]^-{} & \bcE
  }
  $$
be two functors from $\lambda $ to $\bcE $, both commuting with colimits, and let
  $$
  \psi :U\to V
  $$
be a natural transformation. For any ordinal $\alpha <\lambda $, such that $\alpha +1<\lambda $, let
  $$
  \diagram
  U_{\alpha } \ar[dd]_-{\psi _{\alpha }} \ar[rr]^-{} & & U_{\alpha +1} \ar[dd]^-{} \\ \\
  V_{\alpha } \ar[rr]^-{} & & W_{\alpha }
  \enddiagram
  $$
be a push-out square, and let $h_{\alpha }$ be a universal morphism from the colimit $W_{\alpha }$ to $V_{\alpha +1}$. Assume that for any $\alpha <\lambda $, such that $\alpha +1<\lambda $, the morphism $h_{\alpha }$ and the morphism $\psi _0$ are cofibrations in $\bcE $. Then the universal morphism
  $$
  \colim (\psi ):\colim (U)\to \colim (V)
  $$
is also a cofibration in $\bcE $.
\end{lemma}

\begin{pf}
For any ordinal $\alpha <\lambda $, such that $\alpha +1<\lambda $, let $D_{\alpha }$ be the diagram
  $$
  \xymatrix{
  U_0 \ar[r]^-{} \ar[dd]_-{} & U_1 \ar[r]^-{} \ar[dd]_-{}
  & \; ... \ar[r]^-{} & U_{\alpha } \ar[r]^-{} \ar[dd]_-{}
  & U_{\alpha +1} \ar[r]^-{} & \; ... \ar[r]^-{} & \colim (U)
  \\ \\
  V_0 \ar[r]^-{} & V_1 \ar[r]^-{} & \; ... \ar[r]^-{} & V_{\alpha } & & &
  }
  $$
Let also $D_{-1}$ be the diagram
  $$
  U_0\to U_1\to \dots \to U_{\alpha }\to \dots
  \to \colim (U)\; ,
  $$
and let $D_{\lambda }$ be the diagram
  $$
  \xymatrix{
  U_0 \ar[r]^-{} \ar[dd]_-{} & U_1 \ar[r]^-{} \ar[dd]_-{}
  & \; ... \ar[r]^-{} & U_{\alpha } \ar[r]^-{} \ar[dd]_-{}
  & U_{\alpha +1} \ar[r]^-{} \ar[dd]_-{} & \; ... \ar[r]^-{} & \colim (U) \ar[dd]_-{}
  \\ \\
  V_0 \ar[r]^-{} & V_1 \ar[r]^-{}
  & \; ... \ar[r]^-{} & V_{\alpha } \ar[r]^-{}
  & V_{\alpha +1} \ar[r]^-{} & \; ... \ar[r]^-{} & \colim (V)
  }
  $$
Let now $S_{\alpha }=\colim (D_{\alpha })$, $S_{-1}=\colim (D_{-1})=\colim (U)$ and $S_{\lambda }=\colim (D_{\lambda })=\colim (V)$. One has a transfinite filtration of diagrams
  $$
  D_{-1}\subset D_0\subset D_1\subset \dots \subset D_{\alpha }\subset D_{\alpha +1}\subset \dots \subset D_{\lambda }\; .
  $$
Respectively, we have a decomposition of the morphism $\colim (\psi )$ into a transfinite composition
  $$
  \colim (U)=S_{-1}\to S_0\to S_1\to \dots \to S_{\alpha }\to S_{\alpha +1}\to \dots \to S_{\lambda }=\colim (V)\; .
  $$
For any $\alpha <\lambda $, such that $\alpha +1<\lambda $, the square
  $$
  \diagram
  W_{\alpha } \ar[dd]_-{} \ar[rr]^-{h_{\alpha }} & &
  V_{\alpha +1} \ar[dd]^-{} \\ \\
  S_{\alpha } \ar[rr]^-{} & & S_{\alpha +1}
  \enddiagram
  $$
is push-out. Since our input is that all $h_{\alpha }$ and $\psi _0$ are cofibrations, we get that the morphism $\colim (\psi )$ is a transfinite composition of cofibrations in $\bcE $. Since a transfinite composition of cofibrations is a cofibration, the lemma is proved.
\end{pf}

Now we are ready to give the proof of Theorem \ref{F-GSW}. We will only consider the strong symmetrizability case. The symmetrizability assertion then follows by applying in addition the colimit under the action of the symmetric group.

Let $f',f_2,\dots ,f_l$ be $l$ morphisms in $M $ and let $f$ be a push-out of $f'$. To prove (A) we need to show that the morphism $F(f^{\Box n}\Box f_2^{\Box n_2}\Box \dots \Box f_l^{\Box n_l})$ is a cofibration in the category $\bcD ^{\Sigma _n\times \Sigma _{n_2}\times \dots \times \Sigma _{n_l}}$ for any multidegree $\{n,n_2,\dots ,n_l\}$.

By Lemma \ref{vasnecov}, $f^{\Box n}\Box f_2^{\Box n_2}\Box \dots \Box f_l^{\Box n_l}$ is a push-out of $f'^{\Box n}\Box f_2^{\Box n_1}\Box \dots \Box f_l^{\Box n_l}$. Since $F$ commutes with colimits, the morphism $F(f^{\Box n}\Box f_2^{\Box n_2}\Box \dots \Box f_l^{\Box n_l})$ is a push-out of $F(f'^{\Box n}\Box f_2^{\Box n_1}\Box \dots \Box f_l^{\Box n_l})$. Since the latest morphism is a cofibration in $\bcD ^{\Sigma _n\times \Sigma _{n_2}\times \dots \times \Sigma _{n_l}}$, the morphism $F(f^{\Box n}\Box f_2^{\Box n_2}\Box \dots \Box f_l^{\Box n_l})$ is a cofibration too. So, (A) is done.

To prove (B) we just notice that a retract of a cofibration is a cofibration, and retraction is a categoric property coomuting with colimits. This gives (B).

Let $f,g,f_2,\dots ,f_l$ be $l+1$ morphisms in $M $, where $f$ and $g$ are composable. To prove (C) we need to show that for any multidegree $\{n,n_2,\dots,n_l\}$ the morphism $F((gf)^{\Box n}\Box f_2^{\Box n_2}\Box \dots \Box f_l^{\Box n_l})$ is a cofibration in $\bcD ^{\Sigma _n\times \Sigma _{n_2}\times \dots \times \Sigma _{n_l}}$.

By Lemma \ref{at the core} we have that $(gf)^{\Box n}$ is a composition of push-outs of the morphisms $\cor _{\Sigma _{i}\times \Sigma _{n-i}}^{\Sigma _n}(g^{\Box i}\Box f^{\Box (n-i)})$ for $i=0,1,\dots ,n$. By Lemma \ref{kotofei} and Lemma~\ref{vasnecov}, the morphism $(gf)^{\Box n}\Box f_2^{\Box n_2}\Box \dots \Box f_l^{\Box n_l}$ is a composition of push-outs of the morphisms
  $$
  \cor _{\Sigma _{i}\times \Sigma _{n-i}}^{\Sigma _n}(g^{\Box i}\Box f^{\Box (n-i)})
  \Box f_2^{\Box n_2}\Box \dots \Box f_l^{\Box n_l}\; .
  $$
By Lemma~\ref{doublecor}, the latest morphism can be also viewed as the morphism
  $$
  \cor _{\Sigma _{i}\times \Sigma _{n-i}\times \Sigma _{n_2}\times \dots \times \Sigma _{n_l}}^{\Sigma _n\times \Sigma _{n_2}\times \dots \times \Sigma _{n_l}}
  (g^{\Box i}\Box f^{\Box (n-i)}\Box f_2^{\Box n_2}\Box \dots \Box f_l^{\Box n_l})\; .
  $$
Since any $\cor $ is a colimit and the functor $F$ commutes with colimits, the morphism $F((gf)^{\Box n}\Box f_2^{\Box n_2}\Box \dots \Box f_l^{\Box n_l})$ is a composition of push-outs of morphisms of type
  $$
  \cor _{\Sigma _{i}\times \Sigma _{n-i}\times \Sigma _{n_2}\times \dots \times \Sigma _{n_l}}^{\Sigma _n\times \Sigma _{n_2}\times \dots \times \Sigma _{n_l}}
  (F(g^{\Box i}\Box f^{\Box (n-i)}\Box f_2^{\Box n_2}\Box \dots \Box f_l^{\Box n_l}))\; .
  $$
Since the morphisms $f,g,f_2,\dots ,f_l$ are taken from the class $M $, every morphism $F(g^{\Box i}\Box f^{\Box (n-i)}\Box f_2^{\Box n_2}\Box \dots \Box f_l^{\Box n_l})$ is a cofibration in the category $\bcD ^{\Sigma _{n-i}\times \Sigma _i\times \Sigma _{n_2}\times \dots \times \Sigma _{n_l}}$. As $\cor ^G_H$ is a left Quillen functor for any group $G$ and a subgroup $H$ in it, we obtain that $F((gf)^{\Box n}\Box f_2^{\Box n_2}\Box \dots \Box f_l^{\Box n_l})$ is a cofibration in the category $\bcD ^{\Sigma _n\times \Sigma _{n_2}\times \dots \times \Sigma _{n_l}}$.

Now we prove (D). For any ordinal $\lambda $ let ${\rm D}(\lambda )$ be the property (D) in the statement of the theorem being considered for this ordinal $\lambda $. We need to show ${\rm D}(\lambda )$ for any ordinal $\lambda $. To do that we are going to apply the method of transfinite induction. Namely, suppose that for any ordinal $\alpha <\lambda $ the property ${\rm D}(\alpha )$ is satisfied. We will show that this assumption implies that ${\rm D}(\lambda )$ holds true.

Consider a finite collection $f_2,\dots ,f_l$ of morphisms in $M $. We need to show that for any positive integers $n,n_2,\dots ,n_l$ the morphism $F(f_{\infty }^{\Box n}\Box f_2^{\Box n_2}\Box \dots \Box f_l^{\Box n_l})$
is a cofibration in the category $\bcD ^{\Sigma _n\times \Sigma _{n_2}\times \dots \times \Sigma _{n_l}}$. If, for short, we denote the morphism $f_2^{\Box n_2}\Box \dots \Box f_l^{\Box n_l}$ by $e:A\to B$ then we need to show that for any positive integer $n$ the morphism $F(f_{\infty }^{\Box n}\Box e)$
is a cofibration in $\bcD ^{\Sigma _n\times \Sigma _{n_2}\times \dots \times \Sigma _{n_l}}$.

Our strategy is to apply Lemma~\ref{ladder} to the category $\bcE=\bcD ^{\Sigma _n\times \Sigma _{n_2}\times \dots \times \Sigma _{n_l}}$, the functors $U=F(\Box (f_{\alpha,0}^{\Box n},e))$, $V=F(X_{\alpha}^{\wedge n}\wedge B)$, and the natural transformation $\psi=F(f_{\alpha,0}^{\Box n}\Box e)$. First we show that $\colim (\psi )$ is nothing but the morphism $F(f_{\infty }^{\Box n}\Box e)$. This is provided by Lemma \ref{3na}, which says that $f^{\Box n}_{\infty }=\colim (f^{\Box n}_{\alpha ,0})$, the commutativity of the functor $F$ with colimits, and the obvious fact that the right $\Box $-multiplication is colimit-commutative too:
  $$
  \colim (\psi )=\colim F(f^{\Box n}_{\alpha ,0}\Box e)=
  $$
  $$
  F(\colim (f^{\Box n}_{\alpha ,0}\Box e))=
  F(\colim (f^{\Box n}_{\alpha ,0})\Box e)=
  F(f^{\Box n}_{\infty }\Box e)\; .
  $$

Next, we have that
  $$
  \psi _0=F(f_{0,0}^{\Box n}\Box e)=F(\id _{X^{\wedge n}}\Box e)=F(\id _{X^{\wedge n}\wedge B})=\id _{F(X^{\wedge n}\wedge B)}
  $$
is a cofibration in $\bcD ^{\Sigma _n\times \Sigma _{n_2}\times \dots \times \Sigma _{n_l}}$. In order to apply Lemma~\ref{ladder} it remains only to show that the universal morphisms $h_{\alpha}$ are cofibrations in $\bcD ^{\Sigma _n\times \Sigma _{n_2}\times \dots \times \Sigma _{n_l}}$. We give an explicit description of $h_{\alpha}$.

Let $r_{\alpha }$ be a push-out of the morphism $f_{\alpha,0}^{\Box n}$ with respect to the universal morphism between colimits $\Box ^n_{n-1}(f_{\alpha,0})\to \Box ^n_{n-1}(f_{\alpha+1,0})$. Applying Lemma \ref{ulitka} to the corresponding push-out square and the morphism $e:A\to B$
we get a push-out square
  $$
  \diagram
  \Box (f_{\alpha ,0}^{\Box n},e) \ar[dd]_-{f_{\alpha ,0}^{\Box n}\Box e}
  \ar[rr]^-{} & & \Box (r_{\alpha },e) \ar[dd]^-{r_{\alpha }\Box e} \\ \\
  X_{\alpha }^{\wedge n}\wedge B \ar[rr]^-{} & & R_{\alpha }\wedge B
  \enddiagram
  $$
Let furthermore $s_{\alpha }$ be the universal morphism from the colimit $R_{\alpha }$ into the wedge-power $X_{\alpha+1}^{\wedge n}$, \label{See D5}so that $f_{\alpha+1,0}^{\Box n}=s_{\alpha }\circ r_{\alpha }$. Applying Lemma \ref{potap} to this composition and the morphism $e:A\to B$, we obtain yet another push-out square \label{See D6}
  $$
  \diagram
  \Box (r_{\alpha },e) \ar[dd]_-{r_{\alpha }\Box e}
  \ar[rr]^-{} & & \Box (f_{\alpha+1 ,0}^{\Box n},e) \ar[dd]^-{\varkappa _{\alpha }} \\ \\
  R_{\alpha }\wedge B \ar[rr]^-{} & & \Box (s_{\alpha },e)
  \enddiagram
  $$
Composing these two push-out squares we obtain that
  $$
  f_{\alpha+1 ,0}^{\Box n}\Box e=
  (s_{\alpha }\Box e)\circ \varkappa _{\alpha }\; .
  $$
This proves that $W_{\alpha}$ from Lemma~\ref{ladder} equals $F(\Box(s_{\alpha},e))$ and $h_{\alpha}$ equals $F(s_{\alpha}\Box e)$ since $F$ commutes with colimits.

By Lemma \ref{at the core}, the morphism $s_{\alpha }$ is the composition
  $$
  f_{\alpha +1,\alpha }^{\Box n}\circ \psht (\cor ^{\Sigma _n}_{\Sigma _{n-1}\times \Sigma _{1}}(f_{\alpha +1,\alpha }^{\Box (n-1)}\Box f_{\alpha ,0}))\circ
  \dots \circ \psht (\cor ^{\Sigma _n}_{\Sigma _{1}\times \Sigma _{n-1}}(f_{\alpha +1,\alpha }\Box f_{\alpha ,0}^{\Box (n-1)}))\; .
  $$
By Lemma \ref{kotofei} the morphism $s_{\alpha }\Box e$ is the composition of push-outs of the morphisms
  $$
  \psht (\cor ^{\Sigma _n}_{\Sigma _{i}\times \Sigma _{n-i}}(f_{\alpha +1,\alpha }^{\Box i}\Box f_{\alpha ,0}^{\Box (n-i)}))\Box e\; ,
  $$
where $0=1,\dots ,n-1$. By Lemma \ref{ulitka},
  $$
  \psht (\cor ^{\Sigma _n}_{\Sigma _{i}\times \Sigma _{n-i}}(f_{\alpha +1,\alpha }^{\Box i}\Box f_{\alpha ,0}^{\Box (n-i)}))\Box e=
  \psht (\cor ^{\Sigma _n}_{\Sigma _{i}\times \Sigma _{n-i}}(f_{\alpha +1,\alpha }^{\Box i}\Box f_{\alpha ,0}^{\Box (n-i)})\Box e)\; .
  $$
Since $e=f_2^{\Box n_2}\Box \dots \Box f_l^{\Box n_l}$, by Lemma \ref{doublecor} we have that
  $$
 \psht (\cor ^{\Sigma _n}_{\Sigma _{i}\times \Sigma _{n-i}}(f_{\alpha +1,\alpha }^{\Box i}\Box f_{\alpha ,0}^{\Box (n-i)}))\Box e=
 $$
  $$
  \cor ^{\Sigma _n\times \Sigma _{n_2}\times \dots \times \Sigma _{n_l}}_{\Sigma _{i}\times \Sigma _{n-i}\times \Sigma _{n_2}\times \dots \times \Sigma _{n_l}}(f_{\alpha +1,\alpha }^{\Box i}\Box f_{\alpha ,0}^{\Box (n-i)}\Box f_2^{\Box n_2}\Box \dots \Box f_l^{\Box n_l})\; .
  $$
Since $F$ commutes with colimits, it follows that for any ordinal $\alpha $, such that $\alpha +1<\lambda $, the morphism $h_{\alpha }$ is a composition of push-outs of the morphisms
  $$
  F(f_{\alpha +1,\alpha }^{\Box i}\Box f_{\alpha ,0}^{\Box (n-i)}\Box f_2^{\Box n_2}\Box \dots \Box f_l^{\Box n_l})\; ,
  $$
where $i=0,\dots ,n-1$. By the inductive hypothesis, any such morphism is a cofibration. Then $h_{\alpha }$ is a cofibration too. As we have shown above, $F(f_{\infty}^{\Box n}\Box e)=\colim (\psi)$. By Lemma \ref{ladder}, this morphism is a cofibration in $\bcD ^{\Sigma _n\times \Sigma _{n_2}\times \dots \times \Sigma _{n_l}}$. This finishes the proof of Theorem \ref{F-GSW}.

\section{K\"unneth towers for cofibre sequences}

Here we prove the existence of special towers of cofibrations connecting symmetric powers in cofibre sequences via the K\"unneth rule, provided (trivial) cofibrations are symmetrizable, Theorem~\ref{kuennethkey}. This suggests to introduce the concept of a categorified $\lambda $-structure in $\bcC $ and $Ho(\bcC )$. Using the results from Section~\ref{symcofib}, we prove the existence of the $\lambda $-structure of left derived symmetric powers provided symmetrizability of generating (trivial) cofibrations in $\bcC $, Theorem~\ref{derived} and Corollary~\ref{kuennethtriang}. An application to categorical finite-dimensionality (with coefficients in $\ZZ $) is given in Corollary~\ref{additivity}.

In a model category $\bcD $, if $X\to Y$ is a cofibration, then let $Y/X$ be the colimit of the diagram $Y\leftarrow X\to *$, and if $X$ and $Y$ are cofibrant, then $X\to Y\to Y/X$ is a cofibre sequence in $\bcD $.

\begin{theorem}
\label{kuennethkey}
Let $\bcC $ be a closed symmetric monoidal model category, and let $X\stackrel{f}{\to }Y\to Z$ be a cofibre sequence in $\bcC $. Then, for any two natural numbers $n$ and $i$, $i\leq n$, there is a cofibration $\Box ^n_{i-1}(f)\to \Box ^n_i(f)$ and a $\Sigma _n$-equivariant isomorphism
  $$
  \Box ^n_i/\Box ^n_{i-1}\simeq \cor ^{\sg _n}_{\sg _{n-i}\times \sg _i}(X^{\wedge(n-i)}\wedge Z^{\wedge i})
  $$
in $\bcC $. If $f$ is a symmetrizable cofibration and all symmetric powers $\Sym ^i(X)$ are cofibrant, then the morphism $\tilde \Box ^n_{i-1}(f)\to \tilde \Box ^n_i(f)$, obtained by passing to the colimit of the action of the symmetric group $\Sigma _n$, is a cofibration, and $\tilde \Box ^n_i/\tilde \Box ^n_{i-1}$ can be computed by K\"unneth's rule,
  $$
  \tilde \Box ^n_i/\tilde \Box ^n_{i-1}\simeq
  \Sym ^{n-i}(X)\wedge \Sym ^i(Z)\; .
  $$
If $f$ is a symmetrizable trivial cofibration, then all the cofibrations $\tilde \Box ^n_{i-1}(f)\to \tilde \Box ^n_i(f)$ are trivial cofibrations.
\end{theorem}

\begin{pf}
The proof is similar to the proof of Lemma~\ref{at the core}. For any $0\le i\le n$ the diagram $X^{\wedge (n-i)}\wedge K^i_{i-1}(f)$ is a subdiagram in $K^n_{i-1}(f)$. Since the wedge product commutes with colimits, we obtain a universal morphism from $X^{\wedge (n-i)}\wedge \Box^i_{i-1}(f)$ to $\Box^n_{i-1}(f)$. Since $X^{\wedge (n-i)}\wedge Y^{\wedge i}$ is a vertex in the diagram $K_i^n(f)$, we have a morphism from $X^{\wedge (n-i)}\wedge Y^{\wedge i}$ to $\Box_i^n(f)$. Finally, we have a standard morphism $X^{\wedge (n-i)}\wedge \Box^i_{i-1}(f)\to X^{\wedge(n-i)}\wedge Y^{\wedge i}$. Collecting these morphisms together we get a commutative diagram
  $$
  \diagram
  X^{\wedge (n-i)}\wedge \Box^i_{i-1}(f) \ar[dd]_-{}\ar[rr]^-{}& & X^{\wedge(n-i)}\wedge Y^{\wedge i} \ar[dd]_-{}\\ \\
  \Box^n_{i-1}(f)   \ar[rr]^-{} & & \Box^n_i(f)
  \enddiagram
  $$
This is a $\Sigma_{n-i}\times \Sigma_{i}$-equivariant commutative diagram, which yields a $\Sigma_n$-equivariant commutative diagram
  $$
  \diagram
  \cor ^{\Sigma _n}_{\Sigma _{n-i}\times \Sigma _{i}}(X^{\wedge (n-i)}\wedge \Box^i_{i-1}(f)) \ar[dd]_-{} \ar[rr]^-{} & & \cor ^{\Sigma _n}_{\Sigma _{n-i}\times \Sigma _{i}}(X^{\wedge(n-i)}\wedge Y^{\wedge i}) \ar[dd]_-{}\\ \\
  \Box^n_{i-1}(f)   \ar[rr]^-{} & & \Box^n_i(f)
  \enddiagram
  $$
Straightforward verification shows that this is a push-out square.

By an axiom of a closed symmetric monoidal model category,
the morphism $\Box^i_{i-1}(f)\to Y^{\wedge i}$ is a cofibration and we have
  $$
  Y^{\wedge i}/\Box^i_{i-1}(f)\simeq Z^{\wedge i}\;.
  $$
By the same axiom, the functor $X^{\wedge (n-i)}\wedge-$ commutes with colimits and preserves cofibrations in $\bcC$ as the object $X$ is cofibrant. Also the same is true for the functor $\cor$, because this is a bouquet in the category $\bcC$. This implies the needed statements about $\Box^n_i(f)$.

Now suppose that $f$ is a symmetrizable (trivial) cofibration. Recall that taking a quotient over ${\Sigma_n}$ commutes with colimits being a left adjoint functor. This gives a push-out square
  $$
  \diagram
  \Sym ^{n-i}(X)\wedge \tilde \Box ^i_{i-1}(f) \ar[dd]_-{} \ar[rr]^-{} & & \Sym ^{n-i}(X)\wedge \Sym ^i(Y)\ar[dd]_-{}
  \\ \\
  \tilde \Box^n_{i-1}(f) \ar[rr]^-{} & & \tilde \Box ^n_i(f)
  \enddiagram
  $$
The symmetric power $\Sym ^{n-i}(X)$ is cofibrant by assumption. The morphism $\tilde \Box ^i_{i-1}(f)\lra \Sym ^i(Y)$ is a (trivial) cofibration by assumption. Therefore the top morphism is a (trivial) cofibration. This finishes the proof.
\end{pf}

\begin{corollary}
\label{moscowdust}
Let $f$ be a cofibration between cofibrant objects in $\bcC $. Suppose that $f$ is a symmetrizable cofibration, and all symmetric powers $\Sym ^n(X)$ are cofibrant in $\bcC $. Then $f$ is a symmetrizable trivial cofibration if and only if $\Sym ^n(f)$ is a trivial cofibration for all $n\geq 0$.
\end{corollary}

\begin{pf}
Consider the sequence of cofibrations
  $$
  \Sym ^n(X)=\tilde \Box ^n_0(f)\to \tilde \Box ^n_1(f)\to \dots \to \tilde \Box ^n_i(f)\to \dots \to \tilde \Box ^n_n(f)=
  \Sym ^n(Y)
  $$
provided by Theorem \ref{kuennethkey}. The composition of all the cofibrations in that chain is $\Sym ^n(f)$. If $f$ is a symmetrizable trivial cofibration then each cofibration
  $$
  \tilde \Box ^n_i(g)\lra \tilde \Box ^n_{i+1}(f)
  $$
is a trivial cofibration by Theorem \ref{kuennethkey}. Then so is $\Sym ^n(f)$. Conversely, suppose $\Sym ^n(f)$ is a trivial cofibration for any $n\geq 0$. Let's prove by induction on $n$ that the morphism $\tilde \Box ^n_{n-1}(f)\to \Sym ^n(Y)$ is a trivial cofibration, i.e. that $f$ is a symmetrizable trivial cofibration. The base of induction, $n=1$, is obvious. To make the inductive step we observe that in proving Theorem \ref{kuennethkey} we deduce that $\tilde \Box ^n_{i-1}(f)\to \tilde \Box ^n_i(f)$ is a trivial cofibration by only using that $\tilde \Box ^i_{i-1}(f)\to \Sym ^i(Y)$ is a trivial cofibration for $i<n$. But the last condition holds by the induction hypothesis. Thus, all morphisms $\tilde \Box ^n_{i-1}(f)\to \tilde \Box ^n_i(f)$, $1\leq i\leq n-1$, are trivial cofibrations. Then $\tilde \Box ^n_{n-1}(f)\to \Sym ^n(Y)$ is a weak equivalence by $2$-out-of-$3$ property for weak equivalences. Finally, by the assumption of the lemma, $\tilde \Box ^n_{n-1}(f)\to \Sym ^n(Y)$ is a cofibration, and so a trivial cofibration.
\end{pf}

For any closed symmetric monoidal model category $\bcC $ with monoidal unit $\uno $, a {\it $\lambda $-structure} on $\bcC $ is a sequence $\Lambda $ of endofunctors $\Lambda ^n:\bcC \to \bcC $, $n=0,1,2,\dots $, such that (i) $\Lambda ^0=\uno $, $\Lambda ^1=\Id $, (ii) $\Lambda ^n(\emptyset )=\emptyset $ for all $n\geq 1$, (iii) to each cofibre sequence $X\to Y\to Z$ in $\bcC $ and any $n$ there is associated a unique sequence of cofibrations between cofibrant objects
  $$
  \Lambda ^n(X)=L^n_0\to L^n_1\to \dots \to L^n_i\to \dots \to
  L^n_n=\Lambda ^n(Y)\; ,
  $$
called {\it K\"unneth tower}, such that for each index $0\leq i\leq n$ one has isomorphisms
 $$
 L^n_i/L^n_{i-1}\simeq \Lambda ^{n-i}(X)\wedge \Lambda ^i(Z)\; ,
 $$
and (iv) such towers are functorial in cofibre sequences in the obvious sense. In particular, the endofunctors $\Lambda ^n$ preserve cofibrant objects in $\bcC $. In these terms, Theorem \ref{kuennethkey} says that if cofibrations in $\bcC $ are symmetrizable, then symmetric powers yield a specific $\lambda $-structure in $\bcC $. We will call it the canonical $\lambda $-structure of symmetric powers in $\bcC $.

A cofibre sequence in $Ho(\bcC )$ is a sequence of two composable morphisms, which is isomorphic to a sequence coming from a cofibre sequence in $\bcC $ via the functor from $\bcC $ to $Ho(\bcC )$. A similar definition of a $\lambda $-structure can be then given also in $Ho(\bcC )$. If $\Lambda ^*$ is a $\lambda $-structure on $\bcC $ such that $\Lambda ^n$ takes trivial cofibrations between cofibrant objects into weak equivalences, then by Ken Brown's lemma the left derived functors $L\Lambda ^n$ exist, and their collection gives a $\lambda $-structure in $Ho(\bcC )$. Combining this with Corollary \ref{moscowdust}, we obtain the following important result.

\begin{theorem}
\label{derived}
Let $\bcC $ be a closed symmetric monoidal model category, such that all cofibrations are symmetrizable, and all trivial cofibrations between cofibrant objects are symmetrizable in $\bcC $. Then symmetric powers $\Sym ^n$ take weak equivalences between cofibrant objects to weak equivalences, and the canonical $\lambda $-structure of symmetric powers in $\bcC $ induces the $\lambda $-structure of left derived symmetric powers $\LSym ^n$ in $Ho(\bcC )$.
\end{theorem}

\begin{remark}
\label{pinguin}
{\rm Let $\bcC $ be a closed symmetric monoidal model category cofibrantly generated by a set of generating cofibrations $I$ and a set of generating trivial cofibrations $J$. Suppose $I$ and $J$ are both symmetrizable. Then by Corollaries~\ref{cofad} and~\ref{symcof}, the conditions of Theorem~\ref{derived} are satisfied.}
\end{remark}

Assume now that $\bcC $ is moreover pointed. According to~\cite{Hovey1}, there is a well-defined $S^1$-suspension functor $-\wedge^L{S^1}:\bcT \to \bcT $ provided by a $Ho(\SSets_*)$-module structure on the homotopy category $\bcT=Ho(\bcC)$. If it is an autoequivalence on $\bcT $ then $\bcT $ is triangulated, where the translation functor $[1]$ is given by $-\wedge^LS^1$ and distinguished triangles come from cofibre sequences in $\bcC $, see Chapter 7 in loc.cit. Since $\bcC $ is closed symmetric monoidal, so is the triangulated category $\bcT $, and the functor $\bcC \to \bcT $ is monoidal as well, see Section 4.3, loc.cit. We will denote the monoidal product in $\bcT $ also by $\wedge $. A $\lambda $-structure in $\bcT =Ho(\bcC )$ associates K\"unneth towers to distinguished triangles in $Ho(\bcC )$ in the functorial way. Using Theorem \ref{derived} we obtain the following result.

\begin{corollary}
\label{kuennethtriang}
Let $\bcT $ be the homotopy category of a pointed closed symmetric mono\-idal model category $\bcC $, so that $\bcT $ is triangulated. Assume, furthermore, that all cofibrations are symmetrizable, and all trivial cofibrations between cofibrant objects are symmetrizable in $\bcC $. Then $\bcT $ inherits the canonical $\lambda $-structure of left derived symmetric powers associated to distinguished triangles in $\bcT $.
\end{corollary}

As a straightforward consequence of Corollary \ref{kuennethtriang} we also get the following corollary.

\begin{corollary}
\label{additivity} Let $\bcT $ be as above, and let $X\stackrel{f}{\to }Y\to Z\to X[1]$ be a distinguished triangle in $\bcT $. If there exist natural numbers $n'$ and $m'$ such that $L\Sym ^n(X)=0$ for all $n\ge n'$ and $L\Sym ^m(Z)=0$ for all $m\ge m'$, then there exists $N'$ such that $L\Sym ^N(Y)=0$ for all $N\ge N'$.
\end{corollary}

\section{Localization of symmetric powers}
\label{section-localization}

Our main result in this section is a necessary and sufficient condition for trivial cofibrations to remain symmetrizable after Bousfield localization of $\bcC $ with respect to a set of morphism $S$, Theorem~\ref{localization}. This will be applied to the localization by an abstract interval in Section~\ref{moncylinders}.

Let $\bcC $ be a left proper cellular model category, and denote the model structure in $\bcC $ by $\bcM $. Recall that left properness means that the push-out of a weak equivalence along a cofibration is a weak equivalence. Cellularity means that $\bcC $ is cofibrantly generated by a set of generating cofibrations $I$ and a set of trivial generating cofibrations $J$, the domains and codomains of morphisms in $I$ are all compact relative to $I$, the domains of morphisms in $J$ are all small relative to the cofibrations, and cofibrations are effective monomorphisms. Further details about the notions engaged here can be found in \cite{Hovey1}, \cite{Hovey2} or \cite{Hirsch}. Let $S$ be a set of morphisms in $\bcC $. Recall that an object $Z$ in $\bcC $ is called $S$-local if it is fibrant, and for any morphism $f:A\to B$ in $S$ the morphism between function complexes
  $$
  \map (f,Z):\map (B,Z)\lra \map (A,Z)
  $$
is a weak equivalence in $\SSets $. The description of the bi-functor $\map(-,-)$ can be found, for example, in Chapter 5 of \cite{Hovey1}. A morphism $g:X\to Y$ in $\bcC $ is called an $S$-local equivalence if
  $$
  \map (g,Z):\map (Y,Z)\lra \map (X,Z)
  $$
is a weak equivalence in $\SSets $ for any $S$-local object $Z$ in $\bcC $. Since $\map(-,-)$ is a homotopic invariant, each weak equivalence is an $S$-local equivalence in $\bcC $.

By the main result in \cite{Hirsch}, under the above assumptions there exists a new left proper cellular model structure $\bcM _S$ on $\bcC $ whose cofibrations remain unchanged and new weak equivalences $W_S$ are exactly $S$-local equivalences in $\bcC $. The new model structure is cofibrantly generated by the set of generating cofibrations $I$ and a new set of generating trivial cofibrations $J_S$, and it is called a (left) Bousfield localization of $\bcM $ with respect to $S$. The symbol $\bcC _S$ will be used to denote the same category $\bcC $, endowed with the new model structure $\bcM _S$. Then $\bcC _S$ is a (left) Bousfield localization of $\bcC $ with respect to $S$.

Let $F:\bcC\to\bcD$ be a left Quillen functor such that $F(Q(f))$ is a weak equivalence for any $f\in S$, where $Q$ denotes the cofibrant replacement functor in the model structure $\bcM$. Then $F$ is still left Quillen with respect to the localized model structure $\bcM_S$ and has a left derived with respect to $\bcM_S$, see Proposition 3.3.18(1) in ~\cite{Hirsch}. Our main goal is to construct left derived symmetric powers for the localized model category. Since symmetric powers do not admit right adjoints in general, and thus are not left Quillen, we need to strengthen the above result.

Given a functor $F:\bcC\to \bcD$ to a model category, we say that a morphism $g$ in $\bcC$ is $F$-acyclic if $g$ is a cofibration between cofibrant objects in $\bcC$ and $F(g)$ is a weak equivalence in $\bcD$. Obviously, given composable cofibrations between cofibrant objects, their $F$-acyclicity has $2$-out-of-$3$ property. By an $S$-local cofibration we mean a cofibration which is an $S$-local equivalence in $\bcC $.

\begin{theorem}
\label{Flocalization}
Let $F:\bcC\to \bcD$ be a functor to a model category such that all trivial cofibrations between cofibrant objects in $\bcM$ are $F$-acyclic and $F(Q(f))$ is a weak equivalence in $\bcD$ for any $f\in S$. In addition, suppose that $F$-acyclic morphisms are closed under transfinite compositions and push-outs with respect to morphisms to cofibrant objects. Then all $S$-local cofibrations between cofibrant objects are $F$-acyclic. In particular, by Ken Brown's lemma, the left derived functor $LF:Ho(\bcC _S)\to Ho(\bcD)$ exists and commutes with the localization functor $Ho(\bcC )\to Ho(\bcC _S)$.
\end{theorem}

To prove Theorem \ref{Flocalization} we first need to prove an auxiliary result. Fix a left framing on $\bcC$, see Definition 5.2.7 in~\cite{Hovey1}. Thus, for each cofibrant object $X$ one has the functorial cofibrant replacement $X^*$ of the constant cosimplicial object given by $X$, with respect to the Reedy model structure on the category of cosimplicial objects in $\bcC $. The product $X\wedge K$ in $\bcC $ of $X$ and a simplicial set $K$ is then defined as the product $X^*\wedge K$. For any morphism $g$ in $\bcC $, and a morphism $i$ in $\SSets $, we have their push-out product $g\Box i$. For a non-negative integer $m$ let $i_m:\partial\Delta[m]\to\Delta[m]$ be the embedding of the boundary into the $m$-th simplex.

\begin{lemma}
\label{vulcano}
Let $F$ be as in Theorem~\ref{Flocalization}. Then $F$-acyclic morphisms are closed under taking products with simplicial sets generated by finitely many non-degenerate simplices and push-out products with the embeddings $i_m$.
\end{lemma}

\begin{pf}
Let $g:X\to Y$ be an $F$-acyclic morphism in $\bcC $, and let $K$ be a simplicial set. Let $m$ be the maximal dimension of non-degenerated simplices in $K$, and $n$ be the number of such simplices. We apply induction with respect to the lexicographical order on the set of pairs $(m,n)$. Represent $K$ as a simplicial set obtained by gluing an $m$-dimensional simplex to another simplicial set $K'$ having one simplex less than in $K$, i.e. $i:K'\to K$ is a push-out of $i_m$. By Corollary~5.4.4(1) in~\cite{Hovey1}, the functor $X\wedge -$ is left Quillen. It follows that the morphism $X=X\wedge \Delta [0]\to X\wedge \Delta [m]$ is a trivial cofibration between cofibrant objects, whence it is an $F$-acyclic morphism by the assumption on $F$. Since the same is true for $Y\to Y\wedge \Delta [m]$, we see that the morphism $X\wedge \Delta [m]\to Y\wedge \Delta [m]$ is also $F$-acyclic by $2$-out-of-$3$ property for acyclicity. The morphism $X\wedge \Delta [m]\to \Box (g,i_m)$ is a push-out of $g\wedge \id _{\partial \Delta [m]}$. Then it is $F$-acyclic by the push-out property for acyclicity and the induction. Using $2$-out-of-$3$ property once again, we conclude that $g\Box i_m$ is $F$-acyclic. The obvious commutative diagram
  $$
  \diagram
  \Box(g,i_m)\ar[dd]_-{} \ar[rr]^-{} & & \Box(g,i) \ar[dd]^-{} \\ \\
  Y\wedge \Delta [m] \ar[rr]^-{} & & Y\wedge K
  \enddiagram
  $$
is a push-out square, and all objects in it are cofibrant. Then $g\wedge \id _K=\psht (g\Box i_m)\circ \psht (g\wedge \id _{K'})$. The induction and the push-out property finish the proof of the lemma.
\end{pf}

Using a standard transfinite composition argument and Lemma~\ref{vulcano} one can also show that $F$-acyclic morphisms are closed under products with arbitrary simplicial sets and push-out products with arbitrary cofibrations between simplicial sets, though we do not need this. Now we can prove Theorem~\ref{Flocalization}.

\medskip

\begin{pf}
By Ken Brown's lemma and the assumption of the theorem, $F$ sends weak equivalences between cofibrant objects in $\bcC$ to weak equivalences in $\bcD$. For any morphism $f:A\to B$ of $S$ decompose $Q(f)$ into a cofibration $f':Q(A)\to C$ and a trivial fibration $f'':C\to Q(B)$. Since $f''$ is a weak equivalence between cofibrant objects in $\bcC$, $F(f'')$ is a weak equivalence. Let $S'=\{f'|f\in S\}$. Then all morphisms in $S'$ are $F$-acyclic. Since $\bcM _S=\bcM _{S'}$, without loss of generality, one may assume that all morphisms in $S$ are $F$-acyclic.

Next, let $g:X\to Y$ be an $S$-local cofibration between cofibrant objects in $\bcC $. Let $L_S(g):L_S(X)\to L_S(Y)$ be the fibrant replacement of the morphism $g$ with respect to the localized model structure $\bcM _S$. Then $L_S(g)$ is a weak equivalence between cofibrant objects in $\bcC $ , whence $F(L_S(g))$ is a weak equivalence. This gives that the theorem will be proved as soon as we prove that $X\to L_S(X)$ is $F$-acyclic.

By Theorem~4.3.1 in~\cite{Hirsch}, the morphism $X\to L_S(X)$ is a relative $\Lambda$-cell complex, where $\Lambda$ consists of morphisms that are either trivial cofibrations between cofibrant objects, or being composed with a weak equivalence between cofibrant objects are equal to $f\Box (\partial \Delta [n]\to \Delta [n])$, where $f$ runs $S$. By Lemma~\ref{vulcano} and $2$-out-of-$3$ property, all morphisms in $\Lambda$ are $F$-acyclic and the theorem is proved by the assumptions on~$F$.
\end{pf}

Now we need to investigate when the compatibility between the model and monoidal structures is stable under localization. For that we will use the same idea as in the proofs of Theorems 6.3 and 8.11 in \cite{Hovey2}. Since now we assume that $\bcC $ is a closed symmetric monoidal left proper cellular model category cofibrantly generated by the set of generating cofibrations $I$ and the set of generating trivial cofibrations $J$, such that the domains and codomains of the cofibrations from $I$ are cofibrant. Let also $Q$ be the cofibrant replacement in $\bcC $, and so in $\bcC _S$.

\begin{lemma}
\label{locmonoidal}
The model structure $\bcM _S$ is compatible with the monoidal structure in $\bcC $ if and only if for any $X\in \dom (I)\cup \codom (I)$ and for any $f\in S$ the product $X\wedge Q(f)$ is an $S$-local equivalence.
\end{lemma}

\begin{pf}
If $\bcM _S$ is compatible with the monoidal structure in $\bcC $, then $X\wedge Q(f)$ is an $S$-local equivalence by the axioms of a monoidal model category. Conversely, let $h:X\to Y$ be a cofibration in $I$ and let $g:Z\to U$ be an $S$-local cofibration in $\bcC $. By Corollary 4.2.5 in \cite{Hovey1} all we need to show is that $h\Box g$ is an $S$-local cofibration in $\bcC $. By Theorem 2.2 in \cite{Hovey2}, the functors $X\wedge -$ and $Y\wedge -$ are left Quillen with respect to the localized model structure $\bcM _S$. This is because $X\wedge Q(f)$ is an $S$-local equivalence for any $f$ from $S$, and the same for $Y\wedge Q(f)$. Since $X\wedge -$ is left Quillen and $g:Z\to U$ is an $S$-local cofibration, the morphism $\id _X\wedge g:X\wedge Z\to X\wedge U$ is an $S$-local cofibration. Since trivial cofibrations are stable under pushouts, the pushout $Y\wedge Z\to \Box (h,g)$ is an $S$-local cofibration too. The morphism $\id _Y\wedge g:Y\wedge Z\to Y\wedge U$ is an $S$-local cofibration, because $Y\wedge -$ is left Quillen. Since $\id _Y\wedge g$ is the composition $Y\wedge Z\to \Box (h,g)\stackrel{h\Box g}{\lra }Y\wedge U$, we obtain that $h\Box g$ is an $S$-local equivalence. Moreover, $h\Box g$ is a cofibration since $\bcC $ monoidal model.
\end{pf}

\begin{remark}
\label{chernika}
{\rm An analogous statement is true if one considers two model categories $\bcC $ and $\bcD $, such that $\bcC $ is closed monoidal and $\bcD $ is a $\bcC $-module, see Definition 4.2.18 in \cite{Hovey1}.}
\end{remark}

\begin{theorem}
\label{localization}
Let $\bcC $ and $S$ be such that $\bcM _S$ is compatible with the monoidal structure in $\bcC $, and assume furthermore that all cofibrations are symmetrizable and all trivial cofibrations between cofibrant objects are symmetrizable in $\bcC$. Assume also that for any $f\in S$ and any natural $n$ the morphism $\Sym ^n(Q(f))$ is an $S$-local equivalence. Then all $S$-local cofibrations between cofibrant objects are symmetrizable in $\bcC_S$. The left derived functors $L\Sym ^n$ exist on $Ho(\bcC _S)$, and they commute with the localization functor $Ho(\bcC )\to Ho(\bcC _S)$.
\end{theorem}

\begin{pf}
Let $F$ be the composition of $\Sym^n$ and the localization functor $\bcC\to \bcC_S$ (this is just the identity functor considered as a functors between two different model structures).
Since cofibrations in $\bcC$ are symmetrizable, they are so in $\bcC_S$. By Corollary~\ref{moscowdust} applied to $\bcC_S$, we see that trivial symmetrizable cofibrations between cofibrant objects in $\bcC_S$ are the same as $F$-acyclic morphisms in $\bcC$. So, it is enough to show that $S$-local cofibrations are $F$-acyclic.

By Theorem~\ref{F-GSW} applied to the category $\bcC_S$, $F$-acyclic morphisms are closed under transfinite compositions and under push-outs with respect to morphisms to cofibrant objects (actually, to treat transfinite compositions it is enough to use Lemma~\ref{aux} and Theorem~\ref{kuennethkey}). We conclude by Theorem~\ref{Flocalization}.
\end{pf}

\section{Localization w.r.t. diagonalizable intervals}
\label{moncylinders}

Let us consider more closely the important particular case of the left Bousfield localization contracting an object $A$ into a point. If $A$ is what we call a diagonalizable interval, then, using the results from Section~\ref{section-localization}, we prove that trivial cofibrations (between cofibrant objects) remain symmetrizable in the localized category, Theorem~\ref{locdiagint}. As a consequence, we obtain that left derived symmetric powers exist in the homotopy category of the localized category $\bcC _S$, provided we have them in $Ho (\bcC )$, see Corollary~\ref{corol-locint}. This will be applied in Section~\ref{A1homotopy} to the unstable motivic homotopy theory, where $A$ will be the affine line $\AF ^1$ over a base.

Let $\bcC $ be a closed symmetric monoidal left proper cellular model category $\bcC $ cofibrantly generated by the set of generating cofibrations $I$ and the set of generating trivial cofibrations $J$, such that the domains and codomains of the cofibrations from $I$ are cofibrant. Let $A$ be a cofibrant object, let $\pi :A\to \uno $ be a morphism in $\bcC $, and let
  $$
  S=\{ X\wedge A\stackrel{\id _X\wedge \pi }{\lra }X\; |\; X\in \dom(I)\cup \codom (I)\} \; .
  $$

For any morphism $f:X\to Y$ and any object $Z$ in $\bcC $ the morphism $\iHom (f,Z):\iHom (Y,Z)\to \iHom (X,Z)$ in $\bcC $, as well as the morphism $\map (f,Z):\map (Y,Z)\to \map (X,Z)$ in $\SSets $, will be denoted by $f^*$.

Notice that, if $X\in \dom(I)\cup \codom (I)$, it is cofibrant, and since $A$ is cofibrant, the monoidal product $X\wedge A$ is cofibrant too.

\begin{lemma}
\label{ryzhik}
An object $Z$ in $\bcC $ is $S$-local if and only if $Z$ is fibrant in $\bcC $ and the induced morphism $\pi ^*:Z\simeq \iHom (\uno ,Z)\to \iHom (A,Z)$ is a weak equivalence in $\bcC $.
\end{lemma}

\begin{pf}
Let $X\in \dom (I)\cup \codom (I)$. If $\pi ^*$ is a weak equivalence, the morphism $\map (X,\pi^*):\map (X,Z)\to \map (X,\iHom (A,Z))$ is a weak equivalence of simplicial sets. If $Z$ is fibrant, then the simplicial sets $\map (X,\iHom (A,Z))$ and $\map(X\wedge A,Z)$ are canonically weak equivalent, since the objects $X$ and $A$ are cofibrant in $\bcC $. The composition of the morphism $\map (X,\pi ^*)$ with this weak equivalence equals to the morphism $(\id _X\wedge \pi )^*:\map (X,Z)\to \map (X\wedge A,Z)$, so that $(\id _X\wedge \pi )^*$ is a weak equivalence of simplicial sets as well. By definition, it means that $Z$ is $S$-local. Conversely, if $Z$ is $S$-local, the morphism $(\id _X\wedge \pi )^*$ and so $\map (X,\pi ^*)$ are weak equivalences of simplicial sets. Then $Z\simeq \iHom (\uno ,Z)\stackrel{\pi ^*}{\to }\iHom (A,Z)$ is a weak equivalence in $\bcC $ by Proposition 3.2 in \cite{Hovey2}.
\end{pf}

\begin{lemma}
\label{kolobok}
If $Y$ is a cofibrant object in $\bcC $, the morphism $Y\wedge A\stackrel{\id _Y\wedge \pi }{\lra }Y\wedge \uno \simeq Y$ is an $S$-local equivalence, i.e. a weak equivalence in $\bcC _S$.
\end{lemma}

\begin{pf}
For any $S$-local object $Z$ the morphism $\pi ^*:Z\to\iHom (A,Z)$ is a weak equivalence by Lemma \ref{ryzhik} so that $\map (Y,\pi ^*)$ is a weak equivalence of simplicial sets. As in the proof of Lemma \ref{ryzhik} this implies that $(\id _Y\wedge \pi )^*$ is a weak equivalence of simplicial sets for any $S$-local $Z$. This means that the morphism $Y\wedge A\stackrel{\id _Y\wedge \pi }{\lra }Y$ is an $S$-local equivalence.
\end{pf}

\begin{proposition}
\label{yabloko}
Let $\bcC $ and $S$ be as above. Then the model structure $\bcM _S$ is compatible with the monoidal structure in $\bcC $.
\end{proposition}

\begin{pf}
Let $X$ be an object in $\dom (I)\cup \codom (I)$ and let $f$ be a morphism from the set $S$. By definition, there exists $W\in \dom (I)\cup \codom (I)$, such that $f=\id _W\wedge \pi :W\wedge A\to W$. Smashing with $X$ we obtain the morphism $\id _X\wedge f:X\wedge W\wedge A\to X\wedge W$. Applying Lemma \ref{kolobok} to $Y=X\wedge W$ we obtain that $\id _X\wedge f$ is a weak equivalence in $\bcC $. Hence, the category $\bcC $ and the set $S$ satisfy the conditions of Lemma \ref{locmonoidal}. Notice that the cofibrant replacements can be ignored here because $X$ and $W$ are in $\dom (I)\cup \codom (I)$, so that they are cofibrant, and $A$ is cofibrant too.
\end{pf}

Notice that the proof of Proposition \ref{yabloko} follows closely the proofs of Theorems 6.3 and 8.11 in \cite{Hovey2}.

Our aim is now to apply Theorem \ref{localization} to $\bcC _S$ with $S$ as above. For this we need to impose more conditions on the morphism $\pi $. Suppose we are given with two morphisms $i_0,i_1:\uno \to A$, such that $\pi \circ i_0=\pi \circ i_1=\id _{\uno }$. If $f,g:X\rightrightarrows Y$ are two morphisms from $X$ to $Y$ in $\bcC $, then we say that $f$ and $g$ are $A$-homotopic if there is a morphism $H:X\wedge A\to Y$, such that $H\circ (\id _X\wedge i_0)=f$ and $H\circ (\id _X\wedge i_1)=g$. If $f:X\to Y$ and $g:Y\to X$ are two morphisms in opposite directions, such that $g\circ f$ is $A$-homotopic to $\id _X$ and $f\circ g$ is $A$-homotopic to $\id _Y$, then $f$ and $g$ are mutually inverse $A$-homotopy equivalences in $\bcC $.

Following \cite{MorelVoevodsky}, we will be saying that $\pi $ is an {\it interval} if there exists a morphism $\mu :A\wedge A\to A$, such that $\mu \circ (\id _A\wedge i_0)=i_0\circ \pi $ and $\mu \circ (\id _A\wedge i_1)=\id _A$ as morphisms from $A$ to itself.

\begin{lemma}
\label{Whitehead}
Let $\pi :A\to \uno $ be an interval in $\bcC $. Then, for any cofibrant object $X$ in $\bcC $, the morphism $\id _X\wedge \pi :X\wedge A\to X\wedge \uno \simeq X$ is an $A$-homotopy equivalence in $\bcC $.
\end{lemma}

\begin{pf}
From the definition of an interval, it follows that $(\id _X\wedge \pi )\circ (\id _X\wedge i_0)=\id _X$. Let $H=\id _X\wedge \mu $, where $\mu $ is taken from the definition of an interval for $A$. Then $(X\wedge A)\wedge A\simeq X\wedge (A\wedge A)\stackrel{\id _X\wedge \mu }{\lra }X\wedge A$ is an $A$-homotopy from $(\id _X\wedge i_0)\circ (\id _X\wedge \pi )$ to $\id _{X\wedge A}$.
\end{pf}

We will say that the object $A$, together with the morphisms $i_0,i_1:\uno \to A$, is {\it diagonalizable} if $A$ is a symmetric co-algebra (possibly, without a co-unit), i.e. there exists a morphism $\delta :A\to A\wedge A$, such that the compositions $(\id _A\wedge \delta )\circ \delta $ and $(\delta \wedge \id _A)\circ \delta $ coincide, $\ttt \circ \delta =\delta $, where $\ttt :A\wedge A\to A\wedge A$ is the transposition in $\bcC $, and there are two equalities $\alpha \circ i_0=(i_0\wedge i_0)\circ \xi $ and $\alpha \circ i_1=(i_1\wedge i_1)\circ \xi $, where $\xi $ is the inverse to the obvious isomorphism $\uno \wedge \uno \stackrel{\sim }{\to } \uno $. By co-associativity, we have also the morphisms $\delta _n:A\to A^{\wedge n}$ obtained by iterating $\delta $. The following lemma is a straightforward generalization of Lemma \ref{Key Lemma}, where $\Delta [1]$ is being replaced by an abstract diagonalizable object $A$.

\begin{lemma}
\label{Strong Zero}
Let $A$ be diagonalizable. Then, for any two $A$-homotopic morphisms $f,g:X\rightrightarrows Y$, and for any positive integer $n$, the morphisms $\Sym ^n(f)$ and $\Sym ^n(g)$ are $A$-homotopic in $\bcC $.
\end{lemma}

\begin{example}
{\rm Let $\bcC $ be as above and assume furthermore that $\bcC $ is simplicial, and that the structures are compatible with each other. Consider the functor $\SSets \to \bcC $ sending a simplicial set $K$ into the object $\uno \wedge K$, and the same on morphisms. Let $\pi :A\to \uno $ be the image of the morphism $\de [1]\to \de [0]$ under this functor. Then $\pi $ is a diagonalizable interval in $\bcC $, where the morphism $\mu :\Delta [1]\times \Delta [1]\to \Delta [1]$ is induced by the multiplication $[1]\times [1]\to [1]$.
}
\end{example}

\begin{example}
\label{A1interval}
{\rm Let $B$ be a Noetherian separated scheme of finite Krull dimension, and let $\bcC $ be the category $\deop Pre(\Sm /B)$ of simplicial presheaves on the category of smooth schemes of finite type over $B$ endowed with the stalk-wise model structure with respect to the Nisnevich or \'etale topology. By abuse of notation, denote by $\AF ^1$ the simplicial presheaf represented by the affine line $\AF ^1_B$ over $B$. The monoidal unit $\uno $ is represented by $B$, as a scheme over itself. The structural morphism $\pi :\AF ^1\to \uno $ is then a diagonalizable interval in $\bcC $, where $\mu :\AF ^1\wedge \AF ^1\to \AF ^1$ is the multiplication induced by the fibre-wise multiplication in $\AF ^1_B$, see \cite{MorelVoevodsky}.
}
\end{example}

Now we are ready to prove the main result of this section.

\begin{theorem}
\label{locdiagint}
Let $\bcC $ be a closed symmetric monoidal left proper cellular model category $\bcC $ cofibrantly generated by the set of generating cofibrations $I$ and the set of generating trivial cofibrations $J$, such that the domains and codomains of the cofibrations from $I$ are cofibrant, and the sets $I$ and $J$ are both symmetrizable. Let $A$ be a cofibrant object and let $\pi :A\to \uno $ be a diagonalizable interval in $\bcC $. Let also $S=\{ X\wedge A\stackrel{\id \wedge \pi }{\lra }X\; |\; X\in \dom(I)\cup \codom (I)\} $ be the set of morphisms in $\bcC $. Then all $S$-local cofibrations between cofibrant objects are symmetrizable.
\end{theorem}

\begin{pf}
By Proposition \ref{yabloko} and Theorem~\ref{F-GSW}, $\bcC $ and $S$ satisfy the first two assumptions of Theorem \ref{localization}, so that we only need to show that they satisfy the third assumption of it. By Theorem \ref{derived}, symmetric powers preserve weak equivalences between cofibrant objects in $\bcC $. This is why, for any $f\in S$, the morphism $\Sym ^n(Q(f))$ is an $S$-local equivalence if and only if the morphism $\Sym ^n(f)$ is an $S$-local equivalence in $\bcC $.

Let now $f$ be the morphism $\id _X\wedge \pi :X\wedge A\to X\wedge \uno \simeq X$ is $S$, where $X\in \dom (I)\cup \codom (I)$. Then $f=\id _X\wedge \pi $ is an $A$-homotopy equivalence by Lemma \ref{Whitehead}. By Lemma \ref{Strong Zero}, $\Sym ^n(f)$ is an $A$-homotopy equivalence too. Since $I$ is symmetrizable, $\Sym ^n(X\wedge A)$ and $\Sym ^n(X)$ are cofibrant by Corollary \ref{symcof}, because $X$ and $A$ are cofibrant.

By Proposition \ref{yabloko}, $\id _Y\wedge \pi $ is an $S$-local equivalence for any cofibrant $Y$. This implies that $A$-homotopic morphisms between cofibrant objects are the same in the homotopy category $Ho(\bcC _S)$. Therefore, an $A$-homotopy between cobibrant objects is an $S$-local equivalence in $\bcC $. Summing up, we obtain that $\Sym ^n(f)$ is an $S$-local equivalence in $\bcC $.
\end{pf}

\begin{corollary}
\label{corol-locint}
If the assumptions of Theorem \ref{locdiagint} are satisfied, the left derived functors $L\Sym ^n$ exist on $Ho(\bcC _S)$ and commute with $Ho(\bcC )\to Ho(\bcC _S)$.
\end{corollary}

\begin{pf}
Follows from Theorem \ref{locdiagint} and Theorem \ref{derived}.
\end{pf}

\section{Positive model structures on spectra}
\label{sectionpositive}

Now we are going to study symmetric powers in stable categories. The main tool here is the idea of a positive model structure. In this section we will construct positive model structures for abstract symmetric spectra with the usual notion of stable weak equivalences, see Theorem~\ref{positive}. Positive model structures will be used in Section~\ref{rybakit}.

Let $\bcC $ be a closed symmetric monoidal model category which is, moreover, left proper and cellular model category. Suppose in addition that all domains of the generating cofibrations in $I$ are cofibrant. Let $T$ be a cofibrant object in $\bcC $. As it was shown in \cite{Hovey2}, with the above collection of structures imposed upon $\bcC $ there is a passage from $\bcC $ to a category
   $$
  \bcS =\Spt ^{\sg }(\bcC ,T)
  $$
of symmetric spectra over $\bcC $ stabilizing the functor
  $$
  -\wedge T:\bcC \lra \bcC \; .
  $$

Let's remind the basics of this construction for reader's sake. Let $\sg $ be a disjoint union of symmetric groups $\Sigma_n$ for all $n\ge 0$, where
$\Sigma_0$ is the permutation of the empty set, so, isomorphic to $\sg_1$, and all groups are considered as one object categories. Let $\bcC ^{\sg }$ be the category of symmetric sequences over $\bcC $, i.e. functors from $\sg$ to $\bcC$. Explicitly, a symmetric sequence is a collection $(X_0,X_1,X_2,\dots)$ of objects in $\bcC$ together with the action of $\sg_n$ on $X_n$ for each $n\geqslant 1$. Since $\bcC $ is closed symmetric monoidal, so is the category $\bcC ^{\sg }$ with the monoidal product given by the formula
  $$
  (X\wedge Y)_n=\vee _{i+j=n}\sg _n\times _{\sg _i\times \sg _j}(X_i\wedge Y_j)\; ,
  $$
where for any group $G$ and a subgroup $H$ in $G$ the functor $G\times _H-$ is the functor $\cor ^G_H$ described in Section \ref{boksikcalculus}, see \cite{HSS} or \cite{Hovey2}. The restriction to the $n$-th slice of the symmetry isomorphism $X\wedge Y\simeq Y\wedge X$ is equal to the product of the right translation
  $$
  \sg _n\to \sg_n\; ,\quad \sigma \mapsto
  \sigma \circ \tau _{j,i}\; ,
  $$
and the symmetry isomorphism $X_i\wedge Y_j\simeq Y_j\wedge X_i$ in $\bcC$, where $\tau_{j,i}$ permutes the first block of $j$ and the second block of $i$ elements, \cite[Sect. 2.1]{HSS}.

Let $S(T)$ be the free commutative monoid on the symmetric sequence $(\emptyset ,T,\emptyset ,\emptyset ,\dots )$, i.e. the symmetric sequence $S(T)=(\uno,T,T^{\wedge 2},T^{\wedge 3},\dots )$, where $\sg _n$ acts on $T^{\wedge n}$ by permutation of factors (recall that $\emptyset$ is the initial object in $\bcC$). Then $\bcS $ is the category of modules over $S(T)$ in $\bcC ^{\sg }$. In particular, any symmetric spectrum $X$ is a sequence of objects $(X_0, X_1,X_2,\dots)$ in $\bcC $ together with $\sg _n$-equivariant morphisms
  $$
  X_n\wedge T\lra X_{n+1}\; ,
  $$
such that for all $n,i\geq 0$ the composite
  $$
  X_n\wedge T^{\wedge i}\lra X_{n+1}\wedge T^{\wedge(i-1)}\to \dots \to X_{n+i}
  $$
is $\sg _n\times \sg _i$-equivariant. One has a natural closed symmetric monoidal structure on $\bcS$ given by product of modules over the commutative monoid $S(T)$.

For any non-negative $n$ consider the evaluation functor
  $$
  \Ev _n:\bcS \lra \bcC
  $$
sending any symmetric spectrum $X$ to its $n$-slice $X_n$. Each $\Ev _n$ has a left adjoint
  $$
  F_n:\bcC \lra \bcS \; ,
  $$
which can be constructed as follows. First we define a naive functor $\tilde F_n$ from $\bcC $ to $\bcC ^{\Sigma }$ taking any object $A$ in $\bcC $ into the symmetric sequence
  $$
  (\emptyset ,\dots ,\emptyset ,\sg _n\times A,\emptyset ,\emptyset ,\dots )\; ,
  $$
in which $\sg _n\times A$ stays on the $n$-th place. On the second stage we set
  $$
  F_n(A)=\tilde F_n(A)\wedge S(T)\; ,
  $$
see \cite[Def. 7.3]{Hovey2}. Then, for any non-negative integer $m$ one has
  $$
  \Ev _m(F_n(A))=
  \Sigma _m\times _{\Sigma _{m-n}}(A\wedge T^{\wedge (m-n)})\; ,
  $$
where $\Sigma _{m-n}$ is embedded into $\Sigma _m$ by permuting the first $m-n$ elements in the set $\{ 1,\dots ,m\} $.

The functors $F_n$ have the following monoidal property: there is a canonical isomorphism $F_p(A)\wedge F_q(B)\simeq F_{p+q}(A\wedge B)$. The restriction to the $m$-th slice of the symmetry isomorphism $F_p(A)\wedge F_q(B)\simeq F_q(B)\wedge F_p(A)$ is the morphism
  $$
  \Sigma _m\times _{\Sigma _{m-p-q}}(A\wedge B\wedge T^{\wedge (m-p-q)})\to \Sigma _m\times _{\Sigma _{m-p-q}}(B\wedge A\wedge T^{\wedge (m-p-q)})
  $$
which is equal to the product of the right translation
  $$
  \sg _m\to \sg _m\;,\quad \sigma\mapsto \sigma\circ \tau_{q,p}\;,
  $$
the symmetry isomorphism $A\wedge B\simeq B\wedge A$ in $\bcC$, and the identity morphism on $T^{\wedge(m-p-q)}$.

The model structure on $\bcS $ is constructed in two steps -- projective model structure coming from the model structure on $\bcC $ and its subsequent Bousfield localization.

Let $I_T=\cup _{n\geq 0}F_n(I)$ and $J_T=\cup _{n\geq 0}F_n(J)$, where $F_n(I)$ is the set of all morphisms of type $F_n(f)$, $f\in I$, and the same for $F_n(J)$. Let also $W_T$ be the set of projective weak equivalences, where a morphism $f:X\to Y$ is a projective weak equivalence in $\bcS $ if and only if $f_n:X_n\to Y_n$ is a weak equivalence in $\bcC $ for all $n\geq 0$. The projective model structure
  $$
  \bcM =(I_T,J_T,W_T)
  $$
is generated by the set of generating cofibrations $I_T$ and the set of generating weak cofibrations $J_T$. As the model structure in $\bcC $ is left proper and cellular, the projective model structure in $\bcS $ is left proper and cellular too, \cite{Hovey2}. Projective fibrations of spectra are level-wise fibrations. The closed monoidal structure on $\bcS$ is compatible with the model structure $\bcM $.

\begin{remark}
\label{remark-Fn}
{\rm By Remark 7.4. in \cite{Hovey2}, each functor $\Ev _m$ has right adjoint. The above formula for $\Ev _m(F_n(A))$ implies that, given a morphism $f$ in $\bcC $, the morphism $\Ev _m(F_n(f))$ is a coproduct of the product of $f$ with a power of $T$. Since $T$ is cofibrant, $\Ev_m(F_n(f))$ is a (trivial) cofibration provided $f$ is so. This is why $\Ev _m$ sends generating (trivial) cofibrations, in the sense of the model structure $\bcM $, to (trivial) cofibrations in the model category $\bcC $. Applying Lemma 2.1.20 in~\cite{Hovey1}, we see that the functors $\Ev _m$ are left Quillen.}
\end{remark}

Let now
  $$
  \zeta_n^A: F_{n+1}(A\wedge T)\to F_n(A)
  $$
be the adjoint to the morphism
  $$
  A\wedge T\to \Ev_{n+1}(F_n(A))=\Sigma_{n+1}\times (A\wedge T)
  $$
induced by the canonical embedding of $\sg_1$ into $\sg_{n+1}$.
For any set of morphisms $U$ let $\dom (U)$ and $\codom (I)$ be the set of domains and codomains of morphisms from $U$, respectively. Let then
  $$
  S=\{ \zeta ^{A}_n\; \mid \; A\in \dom (I)\cup \codom (I)\; ,\; \; n\geq 0\} \; ,
  $$
where $Q$ is the cofibrant replacement in the projective model structure. Then a stable model structure
  $$
  \bcM _S=(I_T,J_{T,S},W_{T,S})
  $$
in $\bcS $ is defined to be the Bousfield localization of the projective model structure with respect to the set $S$. It is generated by the same set of generating cofibrations $I_T$, and by a new set of generating weak cofibrations $J_{T,S}$. Here $W_{T,S}$ is the set of stable weak equivalences, i.e. new weak equivalences obtained as a result of the localization. The condition of Lemma \ref{locmonoidal} is satisfied and the stable model structure is compatible with the monoidal structure on $\bcS $.

The importance of the stable model structure is that the functor $-\wedge T$ is a Quillen autoequivalence of $\bcS$ with respect to this model structure.

An abstract stable homotopy category, in our understanding, is the homotopy category $\bcT $ of the category of symmetric spectra over a closed symmetric monoidal model category $\bcC$ as above, stabilizing a smash-with-$T$ functor for a cofibrant object $T$ in $\bcC $, i.e. the homotopy category of $\bcS $ with respect to stable weak equivalences $W_{T,S}$.

Notice also that by Hovey's result, see \cite{Hovey2}, the homotopy category $\bcT $ is equivalent to the homotopy category of ordinary $T$-spectra provided the cyclic permutation on $T\wedge T\wedge T$ is left homotopic to the identity morphism.

Now we introduce positive model structures on $\bcS $. Let $I_T^+=\cup _{n>0}F_n(I)$, $J_T^+=\cup _{n>0}F_n(J)$ and let $W_T^+$ be the set of morphisms $f:X\to Y$, such that $f_n:X_n\to Y_n$ is a weak equivalence in $\bcC $ for all $n>0$. We call such morphisms positive projective weak equivalences.

\begin{proposition}
\label{thA}
There is a cofibrantly generated model structure on $\bcS$
  $$
  \bcM ^+=(I_T^+,J_T^+,W_T^+)\;,
  $$
called a positive projective model structure. Positive projective fibrations are level-wise fibrations in positive levels. Positive projective cofibrations are projective cofibrations that are also isomorphisms in the zero level.
\end{proposition}

\begin{pf}
We check that the sets $I_T^+$, $J_T^+$ and $W^+_T$ satisfy the conditions of Theorem 2.1.19 in \cite{Hovey1}, so that they generate a model structure. Condition 1 is satisfied automatically. Conditions 2 and 3 are immediately implied by the inclusions $\hbox{$I_T^+$-cell}\subset\hbox{$I_T$-cell}$, $\hbox{$J_T^+$-cell}\subset\hbox{$J_T$-cell}$ and the fact that $\bcM =(I_T,J_T,W_T)$, whence the sets $I_T$, $J_T$ and $W_T$ satisfy the conditions 2 and 3.

Obviously, all morphisms in $\hbox{$J_T^+$-cell}$ are positive level weak equivalences. To check condition 4 it remains only to show that $\hbox{$J_T^+$-cell}\subset \hbox{$I_T^+$-cof}$. The class $I_T^+$-cof is closed under transfinite compositions and push-outs, see the proof of Lemma 2.1.10 on page 31 in \cite{Hovey1}. Thus, it is enough to show that $J_T^+\subset \hbox{$I_T^+$-cof}$, or, equivalently, that $\hbox{$I^+_T$-inj}\subset \hbox{$J^+_T$-inj}$.
Since the functors $(F_n,\Ev_n)$ are adjoint, we get that
  $$
  \hbox{$J^+_T$-inj}=\{ f:X\to Y \; \, \hbox{in} \; \, \bcS \mid \forall n>0 \; \, \Ev _n(f)\; \, \hbox{is a fibration in}\, \bcC \} \; ,
  $$
i.e. the class $J^+_T$-inj is the class of positive level fibrations in $\bcS $. Similarly,
  $$
  \hbox{$I^+_T$-inj}=\{ f:X\to Y \; \, \hbox{in} \; \, \bcS \mid \forall n>0 \; \, \Ev _n(f)\; \, \hbox{is a trivial fibration in}\, \bcC \} \; .
  $$
It follows that $\hbox{$I^+_T$-inj}\subset \hbox{$J^+_T$-inj}$ and condition 4 is done.
Also, we obtain that $\hbox{$J_T^+$-inj} \, \cap \, W_T^+ = \hbox{$I_T^+$-inj }$, which gives conditions 5 and 6.

The structure of fibrations and cofibrations in $\bcM^+$ can be proved using the definition of $I_T^+$, $J_T^+$, left lifting property and the adjunction between $F_n$ and $\Ev _n$.
\end{pf}

\begin{corollary}
\label{corol:adj}
There is a Quillen adjunction
  $$
  (F_1(T)\wedge-,\iHom(F_1(T),-))
  $$
between $\bcM $ and $\bcM ^+$ and a Quillen adjunction $(\Id ,\Id )$ between $\bcM ^+$ and $\bcM $.
\end{corollary}

Let now
  $$
  S^+=\{ \zeta ^{A}_n\; \mid \; A\in \dom (I)\cup \codom (I)\; ,\; \; n>0\} \; ,
  $$
and let
  $$
  \bcM ^+_{S^+}=(I^+_T,J^+_{T,S^+},W^+_{T,S^+})
  $$
be a localization of the positive projective model structure with respect to the above set $S^+$. We call it a positive stable model structure on $\bcS $. Certainly, we can also localize the positive projective model structure by the set $S$ getting an intermediate model structure $\bcM ^+_S=(I^+_T,J^+_{T,S},W^+_{T,S})$.

\begin{lemma}
\label{prop:monoidalpositive}
With respect to the closed monoidal structure on $\bcS $ the model structure $\bcM ^+$ is an $\bcM $-module and the model structure $\bcM ^+_{S^+}$ is an $\bcM _S$-module. In addition, the closed monoidal structure on $\bcS $ defines an adjunction in two variables with respect to both model structures $\bcM ^+$ and $\bcM ^+_{S^+}$ (see Definition 4.2.1 in \cite{Hovey1}).
\end{lemma}

\begin{pf}
The proof of the facts that $\bcM ^+$ is an $\bcM $-module and that we have an adjunction in two variables with respect to $\bcM ^+$ is similar to the proof of Theorem 8.3 in \cite{Hovey2}. Then we use Lemma \ref{locmonoidal} and Remark \ref{chernika}. Namely, the domains and codomains of morphisms in $I_T$ are of the form $F_n(A)$, $n\ge 0$, where $A$ is a domain or a codomain of a morphism in $I$. Morphisms in $S$ have cofibrant domains and codomains. The analogous is true in the positive setup. Now everything follows from the monoidal properties of the functors $F_n$.
\end{pf}

Notice that the unit axiom is not satisfied for the model structure $\bcM ^+$, thus $\bcS $ is not a closed monoidal model category with respect to $\bcM ^+$. Indeed, let $S(T)^+$ denote the spectrum with $S(T)^+_0=\emptyset$ and $S(T)^+_n=S(T)_n$ for $n>0$. Then the natural morphism $S(T)^+\to S(T)$ is a positive cofibrant replacement for the unit in $\bcS$. However, in general $S(T)^+\wedge X\to X$ is not a positive weak equivalence for a positively cofibrant $X$. For example, if $X=F_n(A)$, $n>0$, then a calculation shows that $(S(T)^+\wedge F_n(A))_m=\emptyset$ for $m\le n$ and $(S(T)^+\wedge F_n(A))_m=(S(T)\wedge F_n(A))_m$ for $m>n$. Thus, the morphism in question fails to be a weak equivalence in level $n$.

\begin{lemma}
\label{lemma:positivestable}
Any positive weak equivalence is a stable weak equivalence.
\end{lemma}

\begin{pf}
Let $f:X\to Y$ be a positive weak equivalence. We claim that for any $Z$ in $\bcS$, there is a canonical bijection
  $$
  \Hom_{Ho(\bcM )}(Z\wedge^L F_1(T),X)=
  \Hom_{Ho(\bcM )}(Z\wedge^L F_1(T),Y)\;.
  $$
For this we use Quillen adjunctions from Corollary \ref{corol:adj} and the fact that $R\iHom (F_1(T),f)$ is an isomorphism in $Ho(\bcM )$ as $f$ is an isomorphism in $Ho(\bcM ^+)$.

Let $g:Y\wedge^L F_1(T)\to X$ be a morphism in $Ho(\bcM )$ that corresponds to the morphism ${\rm id}_Y\wedge^L \zeta^{\uno}_0:Y\wedge^L F_1(T)\to Y$ under the above bijection applied to $Z=Y$ (note that $g$ may be not a class of a morphisms in $\bcC$, which is the reason to consider homotopy categories). Then we obtain a commutative diagram
$$
  \diagram
  X\wedge^L F_1(T)
  \ar[dd]_-{{\rm id}_X\wedge^L \zeta^{\uno}_0} \ar[rr]^-{f\wedge^L \id} & & Y\wedge^L F_1(T) \ar[ddll]_{g}\ar[dd]^-{{\rm id}_Y\wedge^L \zeta^{\uno}_0} \\ \\
  X \ar[rr]^-{f} & & Y
  \enddiagram
$$
The commutativity of the lower triangle is by construction of $g$, while commutativity of the upper triangle is checked by applying $f$ and using the above bijection for the case $Z=X$.
Since ${\rm id}\wedge^L \zeta^{\uno}_0$ is an isomorphism in $Ho(\bcM_S)$, we see that $f$ is also an isomorphism in $Ho(\bcM_S)$ with the inverse being $g\circ ({\rm id}_Y\wedge^L \zeta^{\uno}_0)^{-1}$.
\end{pf}

\begin{theorem}
\label{positive}
In the above terms,
  $$
  W_{T,S}=W^+_{T,S^+}=W^+_{T,S}\; .
  $$
\end{theorem}

\begin{pf}
Let's apply Theorem 3.3.20(1)(a) from \cite{Hirsch} to adjunctions from Corollary \ref{corol:adj}. Indeed, the domains and codomains of morphisms in $S$ and $S^+$ are cofibrant in the corresponding model structures and we have $F_1(T)\wedge S\subset S^+$, $S^+\subset S$, whence the conditions of the above theorem are satisfied. Therefore, we obtain the corresponding Quillen adjunctions between Bousfield localizations $\bcM _S$ and $\bcM ^+_{S^+}$.

We claim that these localized Quillen adjunctions are actually equivalences. More precisely, the functors
  $$
  F_1(T)\wedge^L -:Ho(\bcM_S)\to Ho(\bcM^+_{S^+})\;,
  \quad L\Id:Ho(\bcM^+_{S^+})\to Ho(\bcM_{S})
  $$
are quasiinverse. For this it is enough to show that for any (positively) cofibrant $X$ the natural morphism $F_1(T)\wedge X\to X$ is a (positive) stable weak equivalence. This follows from Lemma \ref{prop:monoidalpositive}, because $F_1(T)\to F_0(\uno)$ is a stable weak equivalence.

Since cofibrant objects in $\bcM ^+_{S^+}$ are the same as in $\bcM ^+$, the equivalence $L\Id :Ho(\bcM ^+_{S^+})\to Ho(\bcM _{S})$ sends an object $X$ in $\bcS $ to $Q^+(X)$, where $Q^+$ is the cofibrant replacement in $\bcM ^+$. Therefore a morphism $f:X\to Y$ in $\bcS $ is in $W^+_{T,S^+}$ if and only if $Q^+(f)$ is in $W_{T,S}$. By Lemma \ref{lemma:positivestable}, the natural morphisms $Q^+(X)\to X$ and $Q^+(Y)\to Y$ are in $W_{T,S}$. Consequently, $Q^+(f)$ is in $W_{T,S}$ if and only if $f$ is in $W_{T,S}$, whence we get $W^+_{T,S^+}=W_{T,S}$. his implies that $(\bcM ^+_{S^+})_S=\bcM ^+_{S^+}$. On the other hand, $(\bcM ^+_{S^+})_S=\bcM ^+_S$, because $S^+\subset S$.
\end{pf}

\begin{corollary}
The monoidal structure on $\bcS $ is compatible with the model structure $\bcM ^+_{S^+}$
\end{corollary}

\begin{pf}
By Theorem \ref{positive}, the morphism $F_1(T)\to F_0(\uno)$ is a cofibrant replacement in $\bcM ^+_{S^+}$. The morphism $F_1(T)\wedge X\to F_0(\uno)\wedge X=X$ is a positive stable weak equivalence for any positively cofibrant $X$ by Lemma \ref{prop:monoidalpositive}.
\end{pf}

\begin{remark}
\label{remark-p}
{\rm For a natural $p$ call a $p$-level weak equivalence (fibration) a morphism in $\bcS$ which is a level weak equivalence (fibration) for $n$-slices with $n\geq p$. These two classes of morphisms define a model structure $\bcM ^{\geq p}$ on $\bcS$. Cofibrations in $\bcM ^{\geq p}$ are cofibrations in $\bcM $ which are isomorphisms on $n$-slices with $n<p$. By methods similar to those used above one shows that any $n$-level weak equivalence is a stable weak equivalence. Moreover, stable weak equivalences are obtained by localization of $\bcM ^{\geq p}$ over the set of morphisms $\{ \zeta ^{A}_n\mid A\in \dom (I)\cup \codom (I)\; ,\; n\geq p\}$.}
\end{remark}

\section{Symmetric powers in stable categories}
\label{rybakit}

Using results from Section~\ref{sectionpositive}, we are now going to show that left derived powers exist and coincide with homotopy symmetric powers for abstract symmetric spectra, see Theorem~\ref{mainresult2} below. This will be applied in Section~\ref{A1homotopy} to the motivic stable homotopy category of schemes over a base.

So, let again $\bcC $ be a closed symmetric monoidal left proper cellular model category, $T$ a cofibrant object in $\bcC $, and $\bcS =\Spt ^{\sg }(\bcC ,T)$ the category of symmetric spectra.
To obtain results for symmetric spectra, similar to Theorem \ref{derived} and Corollary \ref{kuennethtriang}, we would require symmetrizability of generating cofibrations in $\bcS $. However, we can unlikely meet such symmetrizability in applications, see Remark~\ref{remark-nosymm} below. Instead, we will be exploring strong $\Ev _n$-symmetrizability for cofibrations in $\bcS $. The phenomenon of strong $\Ev _n$-symmetrizability was first observed in~\cite{EKMM} for topological spectra. However, our proof for the case of abstract spectra is different from the one in loc.cit., and heavily relies on Theorem~\ref{F-GSW}.

\begin{proposition}
\label{mainresult1}
Let $X$ be an object in $\bcS =\Spt ^{\Sigma }(\bcC ,T)$, cofibrant with respect to the positive projective model structure $\bcM ^+$. Then, for any two positive integers $m$ and $n$, the object $(X^{\wedge n})_m$, as an object of the category $\bcC ^{\Sigma _n}$, is cofibrant in the canonical model structure in $\bcC ^{\Sigma _n}$.
\end{proposition}

\begin{pf}
By Corollary \ref{Fsymcof}, we need only to show that $I_T^+$ is a strongly symmetrizable set of $\Ev _m$-cofibrations for all $m>0$. Let $f_1,\dots ,f_l$
be a finite collection of morphisms in $I$. Recall that $I$ is the set of generating cofibrations in the initial cofibrantly generated category $\bcC $. Let also $p_1,\dots ,p_l$ be a collection of $l$ positive integers. We have to show that the morphism
  $$
  \Ev _m((F_{p_1}f_1)^{\Box n_1}\Box \dots \Box (F_{p_l}f_l)^{\Box n_l})
  $$
is a cofibration in $\bcC ^{\Sigma _{n_1}\times \dots \times \Sigma _{n_l}}$ for any multidegree $\{n_1,\dots ,n_l\}$.

Let $r=n_1p_1+\dots +n_lp_l$, $f=f_1^{\Box n_1}\Box \dots \Box f_l^{\Box n_l}$,
and let $A$ and $B$ be the source and target of the morphism $f$. For any non-negative $i$ the functor $F_i$ commutes with colimits since it is left adjoint. This and the monoidal properties of the functors $F_i$ imply that
  $$
  (F_{p_1}f_1)^{\Box n_1}\Box \dots
  \Box (F_{p_l}f_l)^{\Box {n_l}}=
  F_{n_1p_1+\dots +n_lp_l}(f_1^{\Box n_1}\Box \dots \Box f_l^{\Box n_l})=F_r(f)\; .
  $$
Applying $\Ev _m$ one has
  $$
  \Ev _m(F_r(A))=\Sigma _m\times _{\Sigma _{m-r}}(A\wedge T^{\wedge {(m-r)}})
  $$
and
  $$
  \Ev _m(F_r(B))=\Sigma _m\times _{\Sigma _{m-r}}(B\wedge T^{\wedge {(m-r)}})\; ,
  $$
where the group $\Sigma _{n_1}\times \dots \times \Sigma _{n_l}$ acts on $A$ and $B$ naturally, acts identically on $T^{\wedge(m-r)}$, and it acts by right translations on $\sg_m$ being embedded in it as permutations of the blocks in each of the $l$ clusters of blocks, such that the $i$-th cluster contains $n_i$ blocks of $p_i$ elements each one, for $i=1,\dots ,l$. \label{See D8}

The point here is that this action of the group $\Sigma _{n_1}\times \dots \times \Sigma _{n_l}$ on the set $\{ 1,\dots ,m\} $ induces a free action of the same group on the objects $\Ev _m(F_r(A))$ and $\Ev _m(F_r(B))$ because the (right) action of $\Sigma _{n_1}\times \dots \times \Sigma _{n_l}$ on the right cosets of $\Sigma _{m-r}$ in $\Sigma _m$ is free\footnote{it is essential that all $p_i$ are positive}. It follows that the morphism $\Ev_m(F_r(f))$ in $\bcC ^{\Sigma _{n_1}\times \dots \times \Sigma _{n_l}}$ is isomorphic to a bouquet of several copies of the morphism $(\sg_{n_1}\times\dots\times\sg_{n_l})\times (f\wedge T^{\wedge(m-r)})$. Therefore, $\Ev _m(F_r(f))$ is a cofibration in $\bcC ^{\Sigma _{n_1}\times \dots \times \Sigma _{n_l}}$, as required.
\end{pf}

Let now $\bcD $ be a cofibrantly generated model category and let $G$ be a finite group. Then the functor $Y\mapsto Y/G$ from $\bcD ^G$ to $\bcD $ is left Quillen and it has left derived by Theorem 11.6.8 in \cite{Hirsch}. Given $Y$ in $\bcD ^G$, the homotopy quotient $(Y/G)_h$ is the value of this left derived functor at $Y$. In particular, there is a canonical morphism from $(Y/G)_h$ to $Y/G$, which is a weak equivalence when $Y$ is cofibrant in $\bcD ^G$. If $\bcD $ is in addition simplicial, then the homotopy quotient $(Y/G)_h$ is weak equivalent to the Borel construction $(EG\wedge Y)/G$.

\begin{lemma}
\label{lemma-saynotosimplicial}
Let $Y$ be an object in $\bcS ^G$, such that for any positive integer $m$ the object $Y_m$ is cofibrant in the model structure on $\bcC ^G$. Then the canonical morphism $(Y/G)_h\to Y/G$ is a weak equivalence in $\bcM ^+$.
\end{lemma}

\begin{pf}
Consider the positive projective model structure $\bcM ^+$ on the category $\bcS $ and the induced model structure on $\bcS ^G$. Let $Q^G_+(Y)\to Y$ be the cofibrant replacement in $\bcS ^G$. By Remark~\ref{remark-Fn} and Proposition~\ref{thA}, the functors $\Ev _m$ are left Quillen. Lemma 11.6.4 in \cite{Hirsch} implies that the functors $\Ev ^G_m:\bcS ^G\to \bcC ^G$ are also left Quillen. Therefore, the object $\Ev _m(Q^G_+(Y))=Q^G_+(Y)_m$ is cofibrant in $\bcC ^G$ for all $m$. Combining this with the assumption of the lemma, we see that, for all $m>0$, the canonical morphism ${Q^G_+(Y)}_m/G\to Y_m/G$ is a weak equivalence in $\bcC $. As colimits in spectra are term-wise, the canonical morphism $Q^G_+(Y)/G\to Y/G$ is a positive projective weak equivalence.
\end{pf}

Notice that Lemma~\ref{lemma-saynotosimplicial} is also true for the usual projective model structure $\bcM $, and for more general model structures $\bcM ^{\geq p}$ from Remark~\ref{remark-p}.

Let now $\Sym^n(X)_h$ be the $n$-th homotopy symmetric power of $X$, i.e. the homotopy quotient $(X^{\wedge n}/\Sigma _n)_h$. Combining Proposition~\ref{mainresult1} and Lemma~\ref{lemma-saynotosimplicial}, we obtain the following important result.

\begin{theorem}
\label{mainresult2}
Let $X$ be an object in $\bcS =\Spt ^{\Sigma }(\bcC ,T)$, cofibrant with respect to the positive projective model structure $\bcM ^+$. Then, for any non-negative integer $n$ the natural morphism
  $$
  \theta _{X,n}:\Sym ^n_h(X)\lra \Sym ^n(X)
  $$
is a weak equivalence in $\bcM ^+$. Hence, it is also a stable weak equivalence by Theorem \ref{positive}.
\end{theorem}

\begin{corollary}
\label{hilbre}
Symmetric powers preserve positive projective and stable weak equivalences between positively cofibrant objects in $\bcS $.
\end{corollary}

\begin{pf}
The functors $\Sym ^n_h$, being homotopy quotients, preserve positive projective and stable weak equivalences. Then we apply Theorem~\ref{mainresult2}.
\end{pf}

\begin{corollary}
\label{main3}
Let $\bcT $ be the homotopy category of the category of symmetric spectra $\bcS $. The functors $\Sym ^n:\bcS \to \bcS $ have left derived functors $L\Sym ^n : \bcT \to \bcT $, which are canonically isomorphic to the homotopy symmetric powers $\Sym ^n_h$. Besides, the left deried functors $L\Sym ^n$ give a $\lambda $-structure in $\bcT $, which is canonical in the sense of positive stable model structure on symmetric spectra.
\end{corollary}

\begin{pf}
This is a straightforward consequence of Theorem \ref{mainresult2}, Ken Brown's lemma and the fact that homotopy symmetric powers give rise to K\"unneth towers in distinguished triangles.
\end{pf}

\begin{remark}
\label{remark-nosymm}
{\rm In contrast to level-wise strong symmetrizability asserted by Proposition~\ref{mainresult1}, (positive) cofibrations in $\bcS $ are not symmetrizable in general. Indeed, if $f$ is a cofibration in $\bcC $, then symmetrizability of $F_p(f)$ in $\bcS $, for some $p>0$, is equivalent to strong symmetrizability of $f$ in $\bcC $. Then cofibrations are not symmetrizable for spectra of simiplicial sets by Example~\ref{nostrongex}. Furthermore, by a similar argument as in  Corollaries~\ref{hilbre} and~\ref{main3}, one shows that strong symmetrizability of cofibrations in $\bcC $ implies that left derived symmetric powers exist for $\bcC $ and coincide with the corresponding homotopy symmetric powers. By results from Sections~\ref{intop} and~\ref{A1homotopy}, this gives again that cofibrations are not strongly symmetrizable for (pointed) simplicial sets and, as a consequence, for (pointed) motivic spaces (motivic spaces will be considered in Section \ref{A1homotopy} below).}
\end{remark}

\section{Symmetrizable cofibrations in topology}
\label{intop}

Let us illustrate symmetrizability of (trivial) cofibrations in Kelley spaces and simiplicial sets. Recall that the category $\Top $ of all topological spaces is not a closed symmetric monoidal category, as it does not have an internal Hom in it. The right category is the category of Kelley spaces $\Ke $, see Definition 2.4.21(3) in \cite{Hovey1}. It is a closed symmetric monoidal model category with regard to the monoidal product defined by means of the right adjoint to the embedding of $\Ke $ into $\Top $, see Theorem 2.4.23 and Proposition 4.2.11 in loc.cit. The point here is that the realization functor $|\; \, |$ from $\SSets $ to $\Top $ takes its values in $\Ke $ and, moreover, the it is symmetric monoidal left Quillen, as a functor into $\Ke $, see Proposition 4.2.17, loc.cit. It follows that the category $\Ke $ is simplicial. For any non-negative integer $n$ let $\Delta [n]=\Hom _{\Delta }(-,[n])$ be the $n$-th simplex. If $I_s$ is the set of the canonical inclusions $\partial \Delta [n]\hra \Delta [n]$, $n\geq 0$, and $J_s$ is the set of canonical inclusions $\Lambda _i[n]\hra \Delta [n]$, $n>0$, $0\leq i\leq n$, then $I_s$ and $J_s$ are the sets of generating cofibrations and the set of generating trivial cofibrations for the model structure in $\SSets $. Respectively, the sets $|I_s|=I$ and $|J_s|=J$ cofibrantly generate $\Ke $.

\begin{lemma}
\label{symtop2}
{\rm If $f$ is a weak equivalence in $\SSets $ then $\Sym ^n(f)$ is a weak equivalence in $\SSets $ for any $n\geq 0$.
}
\end{lemma}

\begin{pf}
Let $f:X\to Y$ be a weak equivalence in $\SSets $. Since $|\; \, |$ is a left Quillen functor from $\SSets $ to $\Ke $, all simplicial sets are cofibrant and Kelley spaces are fibrant, $|f|$ is a weak equivalence between fibrant-cofibrant objects in $\Ke $. Then $|f|$ is a left homotopy equivalence in the simplicial closed symmetric monoidal model category $\Ke $. Applying Lemma \ref{Key Lemma}, we obtain that $\Sym ^n(|f|)$ is a weak equivalence in $\Ke $ for all $n\geq 0$. Since $|\; \, |$ is monoidal and left adjoint, we have that $\Sym ^n(|f|)$ is the same morphism as $|\Sym ^n(f)|$.
\end{pf}

\begin{proposition}
\label{symtop1}
All (trivial) cofibrations in $\SSets $, and all (trivial) cofibrations in $\SSets _*$ are symmetrizable.
\end{proposition}

\begin{pf}
By Lemma \ref{pointed}, it is enough to prove the proposition in the unpointed case only. For the set of all cofibrations, since the monoidal product and colimits in $\SSets $ are level-wise, it is enough to prove a similar proposition in the category of sets, where cofibrations are injections. This is an easy exercise. For the set of all trivial cofibrations, we apply Lemma \ref{symtop2} together with Corollary \ref{moscowdust}.
\end{pf}

\begin{proposition}
\label{symtop3}
All (trivial) cofibrations in $\Ke $, and all (trivial) cofibrations in $\Ke _*$ are symmetrizable.
\end{proposition}

\begin{pf}
Since $|I_s|=I$, $|J_s|=J$, and $|\;\,|$ is a symmetric monoidal functor commuting with colimits, we see that by Proposition~\ref{symtop1}, $I$ and $J$ are symmetrizable. Thus we conclude by Corollary~\ref{cofad}.
\end{pf}

Since the sets of cofibrations and trivial cofibrations in $\SSets $, $\SSets _*$, $\Ke$, and $\Ke_*$ are symmetrizable, we can apply Theorem \ref{derived} getting $\lambda $-structures of left derived symmetric powers in the corresponding unstable homotopy categories. In the stable setting, when $\bcS =\Spt ^{\Sigma }(\bcC ,T)$, the category $\bcC $ is the category $\SSets _*$ of pointed simplicial sets and $T$ is the simplicial circle $S^1$, i.e. the coequalizer of the two boundary morphisms $\Delta [0]\rightrightarrows \Delta [1]$, then Theorem \ref{mainresult2} and Corollary \ref{hilbre} specialize to the results \cite{EKMM}, III, 5.1, and \cite{MMSS}, 15.5. Corollary \ref{main3} yields the $\lambda $-structure of left derived symmetric powers in the topological stable homotopy category.

\section{Symmetrizable cofibrations in $\AF ^1$-homotopy theory of schemes}
\label{A1homotopy}

Now we are going to apply the main results of the paper to the Morel-Voevodsky homotopy theory of schemes over a base and prove the existence of $\lambda $-structures of left derived symmetric powers in both unstable and stable settings of that theory.

So, let $B$ be a Noetherian separated scheme of finite Krull dimension, $\Sm /B$ the category of smooth schemes of finite type over $B$, and let $Pre(\Sm /B)$ be the category of presheaves of sets on $\Sm /B$, i.e. contravariant functors from $\Sm /B$ to $\Sets $. Let $\bcC $ be the category $\deop Pre(\Sm /B)$ of simplicial presheaves over $B$. Sometimes it is convenient to think of $\bcC $ as the category $Pre(\Sm /B\times \Delta )$ of presheaves of sets on the Cartesian product of two categories $\Sm /B$ and $\Delta $. If $X$ is a smooth scheme over the base $B$, let $\Delta _X[n]$ be a presheaf on $\Sm/B\times \Delta $ sending any pair $(U,[m])$ to the Cartesian product of sets $\Hom _{\Sm /B}(U,X)\times \Hom _{\Delta }([m],[n])$. Then we get a fully faithful embedding of the category of smooth schemes over $B$ into the category of simplicial presheaves, $h:\Sm /B\to \bcC $, sending $X$ to the presheaf $\Delta _X[0]$ represented by $X$, and similarly on morphisms. If $K$ is a simplicial set, i.e. a presheaf of sets on the simplicial category $\Delta $, then it induces another presheaf on $\Sm/B\times \Delta $ by ignoring schemes and sending a pair $(U,m)$ to the value $K_m$ of the functor $K$ on the object $[m]$ in $\Delta $. This gives a functor $\SSets \to \bcC $, which provides a simplicial structure on the category $\bcC $. The symmetric monoidal structure in $\bcC $ is defined section-wise, i.e. for any two simplicial presheaves $X$ and $Y$ the value of their product on $(U,[m])$ is the Cartesian product of the values of $X$ and $Y$ on $(U,[m])$.

Following Jardine, \cite{Jardine1}, we say that a morphism  $f:X\to Y$ in $\bcC $ is a weak equivalence if $f$ induces weak equivalences on stalks of the presheaves $X$ and $Y$, where stalks are taken in the sense of Nisnevich or \'etale topology on the category $\Sm /B$. Let $W$ be the class of all weak equivalences in $\bcC $. Notice that, in spite of that $\bcC $ is a category of simplicial presheaves, the topology is needed to define weak equivalences in $\bcC $ as we use stalks. Let also $I$ be the set of monomorphisms of type $X\hookrightarrow \Delta _U[n]$ for some simplicial presheaf $X$, smooth $B$-scheme $U$ and $n\ge 0$. Fix a cardinal $\beta >2^{\alpha }$, where $\alpha $ is the cardinality of the morphisms in $\Sm /B$. Let $J$ be the set of monomorphisms $X\to Y$, which are weak equivalences and such that the cardinal of the set of $n$-simplices in $Y$ is less than $\beta $ for all $n$. One can show that the class $I$-cell consists of all section-wise monomorphisms of simplicial presheaves. Then $\bcC $ together with the above defined weak equivalences and monomorphisms taken as cofibrations is a simplicial left proper and cellular closed symmetric monoidal model category cofibrantly generated by the set of generating cofibrations $I$ and the set of generating trivial cofibrations $J$. Actually, this is a consequence of a more general result on model structures for simplicial presheaves on a site due to Jardine, see \cite{Jardine1}. Such constructed model structure $\bcM =(I,J,W)$ is called injective model structure in $\bcC $.

As well as in Example \ref{A1interval}, denote by $\AF ^1$ the simplicial motivic space represented by the affine line $\AF ^1_B$ over the base scheme $B$. Then $\AF ^1\to \uno $ is a diagonalizable interval, with the multiplication coming from the multiplication in the fibres of the structural morphism from $\AF ^1_B$ to $B$. The above injective model structure and the set of morphisms $S=\{ X\wedge \AF ^1\stackrel{\id \wedge \pi }{\lra }X\; |\; X\in \dom(I)\cup \codom (I)\} $ satisfy the assumptions of the localization theorem in \cite{Hirsch}. The corresponding left localized model structure $\bcM _{\AF ^1}=(I,J_{\AF ^1},W_{\AF ^1})$ is one of the motivic model structures on $\bcC $, and the corresponding localization $\bcC _{\AF ^1}$ is again a simplicial left proper cellular closed symmetric monoidal model category cofibrantly generated by the same set of generating cofibrations $I$ and the new localized set of generating trivial cofibrations $J_{\AF ^1}$. The category $\bcC _{\AF ^1}$ is called the unstable motivic model category of schemes over the base $B$. Its homotopy category $Ho(\bcC _{\AF ^1})$ is nothing but the unstable motivic homotopy category of schemes over $B$, which we denote by $\MH (B)$.

The following result is the precise statement of Theorem A mentioned in Introduction.

\begin{theorem}
\label{Teorema A}
Let $B$ be a Noetherian scheme of finite Krull dimension, and let $\bcC _{\AF ^1}$ be the unstable motivic model category of schemes over $B$. Then all symmetric powers $\Sym ^n$ preserve weak equivalences in $\bcC _{\AF ^1}$, and the corresponding left derived functors $\LSym ^n$ yield a $\lambda $-structure in $\MH (B)$.
\end{theorem}

\begin{pf}
Since cofibrations in $\bcC $ are coming section-wise from cofibrations simplicial sets, all objects are cofibrant in $\bcC $. By the same reason, and by Proposition \ref{symtop1}, we also have that all cofibrations in $\bcC $ are symmetrizable. The class of trivial cofibrations is symmetrizable too. Indeed, let $f:X\to Y$ be a trivial cofibration $\bcC $. Since stalks of presheaves are colimits commuting with symmetric powers, the morphism $(\Sym ^n(f))_P$ on stalks at a point $P$ is nothing but the $n$-th symmetric power $\Sym ^n(f_P)$ of the morphism $f_P$ induced by $f$ at $P$. So $(\Sym ^n(f))_P$ is a weak equivalence of simplicial sets by Proposition \ref{symtop1}. Since, moreover, $\AF ^1\to \uno $ is a diagonalizable interval and all objects are cofibrant in $\bcC $, we conclude by Theorem \ref{locdiagint} and Theorem \ref{derived}.
\end{pf}

\begin{remark}
{\rm
Theorem \ref{Teorema A} holds true also in the pointed setting by Lemma \ref{pointed}.
}
\end{remark}

Let now $T$ be the motivic $(1,1)$-sphere. Recall that $T$ is the $\wedge $-product of the simplicial circle, i.e. the coequalizer of the two morphisms from $\Delta [0]$ to $\Delta [1]$, and the algebraic group $\Gm $ over $B$ in the pointed category $\bcC _*$. The corresponding category of symmetric spectra $\bcS =\Spt ^{\sg }((\bcC _{\AF ^1})_*,T)$, together with the corresponding stable model structure, is the category of motivic symmetric spectra over the base scheme $B$, and the homotopy category of $\bcS $, with regard to the stable model structure, is nothing but the Morel-Voevodsky motivic stable homotopy category over $B$, see \cite{VoevBerlin} and \cite{Jardine2}. We will denote it by $\MSH (B)$.

The category $\bcS =\Spt ^{\sg }((\bcC _{\AF ^1})_*,T)$ of motivic symmetric spectra has a structure of a simplicial closed symmetric monoidal model category by Hovey's result, \cite{Hovey2}. Moreover, the simplicial suspension $\sg _{S^1}$ induces an autoequivalence in its homotopy category $\MSH (B)$, so that it is a triangulated category (use Section 6.5 in \cite{Hovey1}). Then we see that the results in Proposition \ref{mainresult1}, Theorem \ref{mainresult2}, Corollary \ref{hilbre} and Corollary \ref{main3} hold true for symmetric spectra of simplicial sets and for motivic symmetric spectra uniformly. In other words, we have the following result (Theorem B in Introduction).

\begin{theorem}
\label{Teorema B}
Let $B$ be a Noetherian scheme of finite Krull dimension, and let $T=S^1\wedge \Gm $ be the motivic sphere. Symmetric powers preserve stable weak equivalences between positively cofibrant objects in the category $\Spt ^{\sg }((\bcC _{\AF ^1})_*,T)$ of motivic symmetric spectra over the base $B$. The corresponding left derived symmetric powers $L\Sym ^n$ exist, they are canonically isomorphic to homotopy symmetric powers and give rise to a $\lambda $-structure in $\MSH (B)$.
\end{theorem}

The category $\MSH (B)$, being triangulated, can be $\QQ $-localized getting the $\QQ $-linear triangulated symmetric monoidal category $\MSH (B)_{\QQ }$. Hirschhorn's localization allows to make symmetric spectra into a $\QQ $-linear stable model category, see Definition 3.2.14 in \cite{CD}. One can show that the $\lambda $-structure from Theorem \ref{Teorema B} induces the $\lambda $-structure of symmetric powers with $\QQ $-coefficients defined via idempotents in endomorphism rings, see 3.3.20 in loc.cit. The latest $\lambda $-structure coincides with the system of towers constructed in \cite{triangles}. If now $\MSH (B)_{\QQ }^{\rm c}$ is the full subcategory of compact objects in $\MSH (B)_{\QQ }$, the $\lambda $-structure of $\QQ $-local left derived symmetric powers induces the $\lambda $-structure in the $K$-theory of the triangulated category $\MSH (B)_{\QQ }^{\rm c}$ considered in \cite{zeta}.

\section{Appendix: categorical v.s. geometrical symmetric powers}

One of the numerous differences between the topological homotopy theory and the motivic one is that motivic spaces or spectra are associated with the geometric reality of deeper level. In particular, this leads to the following important phenomenon. For simplicity, let $B$ be the spectrum of a field $k$. Recall that $h$ is the Yoneda type functor from the category $\Sch /k$ of all separated schemes of finite type over $k$ into the category $\bcC =\deop Pre(\Sm /k)$ of simplicial presheaves on the \'etale or Nisnevich site $\Sm /k$. The problem is that, although the functor $h$ obviously commutes with products, it does not commute with colimits. More precisely, let $X$ be a separated scheme of finite type over $k$, such that any finite subset is contained in an affine open subscheme in $X$. Under this assumption the $n$-th symmetric power $\Sym ^n(X)$ exists as an object in $\Sch /k$. Then the $n$-th symmetric power $\Sym ^n(h_X)$ of the motivic space $h_X$, represented by $X$, is not the same as the motivic space $h_{\Sym ^n(X)}$, represented by the symmetric power $\Sym ^n(X)$ of the scheme $X$ over $k$. The comparison of these two objects is a question of critical importance, since its understanding would provide the geometrical meaning to our categorical approach to symmetric powers in the $\AF ^1$-homotopy setup.

If $\Sets $ is the category of sets with discrete topology, the category of presheaves on the site $\Sets $ has one stalk only. This is why, if $X$ is a set and $G$ a finite group acting on $X$, it is easy to show that the canonical map from $h_X/G$ to $h_{X/G}$ is an isomorphism. Working with the $\QQ $-linear motivic symmetric spectra over a base, the homotopy type of $\Sym ^n(h_X)$ will be the same as of $h_{\Sym ^n(X)}$, due to Jardine's transfers, see \cite{Jardine3}. Below we consider some examples, which suggest what exactly the difference between two homotopy types might depend on.

Let again $\bcC $ be the category of simplicial presheaves on $\Sm /k$ with the model structure given by the \'etale topology on schemes. Let $G$ be a finite group acting on a separated scheme $X$ of finite type over $k$. Suppose $X$ can be covered by $G$-invariant affine open subschemes, so that the quotient $X/G$ exists in $\Sch /k$. Let $\alpha :h_X/G\to h_{X/G}$ be the canonical morphism as above. In the case of symmetric powers, $X$ must be the $n$-th power of a scheme and $G$ the symmetric group $\Sigma _n$ acting by permuting factors in $X$. We address the question whether $\alpha $ is a weak equivalence in the model category $\bcC $.

\begin{proposition}
\label{comparison}
Assume that $G$ acts freely on $X$, i.e. the morphism $\pi :X\to X/G$ is \'etale. Then the canonical morphism $\alpha :h_X/G\to h_{X/G}$ is a weak equivalence in~$\bcC $.
\end{proposition}

\begin{pf}
To prove the proposition it is enough to show that $\alpha $ induces an isomorphism on neighbourhoods of points, i.e. on spectra of strictly Henselian rings. So, let $R$ be a strictly Henselian local ring, $\mg $ be the maximal ideal in it and $l=R/\mg $ be the corresponding residue field. All we need is to show that the canonical morphism of sets $\alpha _R:X(R)/G\to (X/G)(R)$ is an isomorphism. Let $\bcA _R$ be the category of \'etale algebras over $R$ and let $\bcA _l$ be the category of \'etale algebras over $l$. As $R$ is Henselian the residue homomorphism $R\to l$ induces an equivalence of categories $\Psi :\bcA _R\to \bcA _l$. Let $f:\Spec (R)\to X/G$ be an element in $(X/G)(R)$. Its preimage, under the morphism $\pi $, is a set of $R$-points of the \'etale $R$-algebra $S$, where $\Spec (S)\to X$ is the pull-back of $f$ with respect to the morphism $\pi :X\to X/G$. Let also $\bar f$ be the precomposition of $f$ with the morphism $\Spec (l)\to \Spec (R)$, and let $\Spec (L)\to X$ be the pull-back of the precomposed morphism $\bar f$ with respect to $\pi $. As $\Psi $ is an equivalence, we have that $\alpha ^{-1}_R(f)=\alpha ^{-1}_l(\bar f)$, i.e. $\alpha ^{-1}_R(f)$ is bijective to $l$-points of the \'etale $l$-algebra $L$. Since $R$ is strictly Henselian, the residue field $l$ is separably closed, whence $L$ is isomorphic to the product of $n$ copies of the field $l$, where $n$ is the order of the finite group $G$, and $G$ acts freely on $\Spec (L)$. Then the quotient of the set of all $l$-points of $X$ by $G$ is identified with $l$-points of $X/G$. Therefore the quotient of the set of all $R$-points of $X$ by $G$ is identified with $R$-points of $X/G$, whence $\alpha _R$ is a bijection.
\end{pf}

The same argument shows that the morphism $\alpha$ is a monomorphism section-wise in the case of the Nisnevich topology on schemes.

\medskip

\bigskip

\begin{small}

\end{small}

\vspace{4mm}

\begin{small}

{\sc Steklov Mathematical Institute, Gubkina str. 8,
119991, Moscow, Russia} {\footnotesize {\it $\; $ E-mail address}: {\tt gorchins@mi.ras.ru}}

\end{small}

\bigskip

\begin{small}

{\sc Department of Mathematical Sciences, University of Liverpool, Peach Street, Liverpool L69 7ZL, England, UK}

\end{small}

\begin{footnotesize}

{\it E-mail address}: {\tt vladimir.guletskii@liverpool.ac.uk}

\end{footnotesize}

\end{document}